%% file: main.tex
\documentclass{article}[12pt]

\usepackage[margin=1.1in]{geometry}

\usepackage[utf8]{inputenc} %
\usepackage[T1]{fontenc}    %
\usepackage[hidelinks,colorlinks,linkcolor=black,citecolor=orange]{hyperref}  
 \hypersetup{
     colorlinks=true,
     linkcolor=blue,
     filecolor=blue,    
     urlcolor=cyan,
     }
\usepackage{placeins}
\usepackage{url}            %

\usepackage{booktabs}       %
\usepackage{amsfonts}       %
\usepackage{nicefrac}       %
\usepackage[expansion=false]{microtype}      %
\usepackage{xcolor}         %
\usepackage[sortcites=true,citestyle=numeric,natbib=true,bibstyle=numeric,maxbibnames=99]{biblatex}

\usepackage{tikz}
\usetikzlibrary{decorations.pathreplacing}
\usetikzlibrary{positioning}

\usepackage{graphicx}
\usepackage{caption}
\usepackage{subcaption}
\usepackage{amsmath,amssymb,amsthm}
\usepackage{bbm}
\usepackage{enumitem}   
\usepackage{capt-of}
\usepackage{wrapfig}
\usepackage{algorithm}
\usepackage{algpseudocode}
\newcommand{\upperR}[1]{\uppercase\expandafter{\romannumeral#1}}
\usepackage{xcolor}                %
\usepackage[most]{tcolorbox}             %

\tcbuselibrary{theorems}

\newtcbtheorem{defn}{Definition}%
{	colframe=blue!30!gray, theorem name, fonttitle=\bfseries,
	colback=blue!5, top=1mm,bottom=1mm, boxrule=0.5pt}{Example}

\newtcbtheorem[use counter=theorem, number within=section]{THM}{Theorem}%
{	colframe=blue!30!gray, fonttitle=\bfseries,
	colback=blue!5,  fontupper=\itshape, top=1mm,bottom=1mm, boxrule=0.5pt}{Theorem}
\usepackage{varwidth}
\newtcbox{\mybox}{colback=blue!5,
	colframe=blue!30!black, center, enhanced, varwidth upper}
\newtcbox{\mymath}[1][]{%
	nobeforeafter, math upper, tcbox raise base,
	enhanced, colframe=blue!30!black,
	colback=blue!5, boxrule=0.5pt, top=1mm,bottom=1mm,
	#1}

\input{math_commands.tex}
\usepackage[normalem]{ulem}

\makeatletter
\@removefromreset{theorem}{section}
\makeatother

\title{Spectral norm bound for the product of \\random Fourier-Walsh matrices
}
\author{
Libin Zhu\thanks{Department of Mathematics, University of Washington, Seattle, WA 98195; \texttt{https://libinzhu.github.io/}. Research of Zhu was supported by NSF CCF-2023166.}
\and
Damek Davis\thanks{Wharton Department of Statistics and Data Science, University of Pennsylvania,
		Philadelphia, PA 19104, USA;
		\texttt{www.damekdavis.com}. Research of Davis supported by an Alfred P. Sloan research fellowship and NSF DMS award 2047637}
\and
Dmitriy Drusvyatskiy\thanks{Department of Mathematics, U. Washington, Seattle, WA 98195; \texttt{www.math.washington.edu/$\sim$ddrusv}.
		Research of Drusvyatskiy was supported by NSF DMS-2306322, NSF DMS-2023166, and AFOSR FA9550-24-1-0092 awards.}
\and 
Maryam Fazel\thanks{Department of Electrical \& Computer Engineering, University of Washington, Seattle, WA 98195; \texttt{people.ece.uw.edu/fazel\_maryam/}. Research supported by NSF TRIPODS II DMS-2023166, CCF 2007036, CCF 2212261, and CCF 2312775.}
}

\date{April 2024}

\begin{document}
\maketitle
\begin{abstract}
We consider matrix products of the form $A_1(A_2A_2)^\top\ldots(A_{m}A_{m}^\top)A_{m+1}$,  where $A_i$ are normalized random Fourier-Walsh matrices. We identify an interesting polynomial scaling regime when the operator norm of the expected matrix product tends to zero as the dimension tends to infinity.
\end{abstract}

\section{Introduction}
Products of Boolean matrices arise often in combinatorics, probability theory, and analysis of boolean functions. In this work, we study a special class of matrix products generated by random boolean matrices. Setting the stage, consider a set of data points $x^{(1)},\ldots,x^{(n)}$ sampled independently and uniformly from the hypercube $\{-1,1\}^d$. For any family $\mathcal{S}$ of subsets of $\{1,\ldots,d\}$, the {\em Fourier-Walsh matrix} $X_{\mathcal{S}}\in \R^{n\times |\S|}$ has as its $(i,S)$ entry the evaluation of the Fourier-Walsh polynomial $\prod_{s \in S} x^{(i)}_s$, for any $i\in [n]$ and $S\subset\S$. Fourier-Walsh polynomials figure prominently in boolean function analysis because they form an orthonormal basis for the $L_2$-space of functions on $\{-1,1\}^d$ with respect to the uniform measure. On the level of matrices, this implies that for any disjoint set families $\S$ and $\S'$, the orthogonality relation holds:
$$\E[X_{\S}^\top X_{\S'}]=0.$$
Going a step further, it is natural to bound the higher-order moments of the normalized matrix $X_{\S}^\top X_{\S'}/n$. For example, it is straightforward to check that the operator norm of the second-moment satisfies:
\begin{equation}\label{eqn:intro_eqn_moment}
\left\|\E\left[\left(\frac{X_{\S}^\top X_{\S'}}{n}\right)\left(\frac{X_{\S}^\top X_{\S'}}{n}\right)^\top\right]\right\|_{\rm op}= \frac{|\S'|}{n}.
\end{equation}
See the short argument in Appendix~\ref{proof:intro_eqn_moment}.
Interestingly, in the regime $|\S'|=o(n)$, the right-hand side of \eqref{eqn:intro_eqn_moment} tends to zero. Consequently, in this regime the operator norm of the second moment tends to zero. 

In this work, we generalize the estimate of the form \eqref{eqn:intro_eqn_moment} to higher order moments and to more general products of random Fourier-Walsh matrices. The following is our main result.

\begin{theorem}[Main result]\label{lemma:feature_product_tech}
Fix collections of sets $\S_1,\ldots,\S_{m+1} \subseteq 2^{[d]}$ and weights 
$w^{(i)}\in \R^{\S_i}_{+}$. Define the matrix product
$$M= A_1^\top \Bigl(A_2 A_2^\top \Bigr) \cdots \Bigl(A_{m} A_{m}^\top \Bigr) A_{m+1},$$
where $A_i=X_{\mathcal{S}_i} \Diag(w^{(i)})$ are the scaled Fourier-Walsh matrices.
Assume  that the following regularity conditions hold for all indices $i,j\in [m+1]$:
\begin{enumerate}
    \item\label{it:deg_bound} {\bf (degree bound)} The inequalities $|S_i|\leq p_i$ hold for all $S_i\in\S_i$,
    \item\label{it:triv_intersec} {\bf (trivial intersection)} Whenever $\S_i$ intersects $\S_j$, the equality $\S_i=\S_j$ holds,
    \item\label{it:small_weights} {\bf (small weights)} The weights satisfy 
    $w^{(i)}= O_d(n^{-1/2}\wedge d^{-p_i/2})$ for all $i\in [m+1]$.
\end{enumerate}
    Then the estimate
    \begin{equation}\label{eqn:labeled_hard_mat}
    \left\|\E M \right\|_{\rm op}
    = O_d \left(d^{p_j}\right)\|w^{(1)}\|_\infty \|w^{(m+1)}\|_\infty,
    \end{equation}
    holds for any index $j \in [m+1]$ such that $\S_{j}$ is distinct from $\S_{1}$ and $\S_{{m+1}}$.
\end{theorem}

Let us make a few comments about the assumptions imposed in the theorem. The degree bound \eqref{it:deg_bound} asserts that the Fourier-Walsh polynomials indexed by sets in $\S_i$  have degree bounded by $p_i$. The trivial intersection condition \eqref{it:triv_intersec} ensures that the set systems $\S_i$ and $\S_j$ do not intersect wildly, thereby controlling the interdependence between the random matrices in the product. Finally, the small weights condition \eqref{it:small_weights} is natural because it ensures that the matrices $A_i$ are bounded in operator norm in expectation.

The most important observation about the bound in \eqref{eqn:labeled_hard_mat} is that the right-hand-side is controlled by $d^{p_j}$ for {\em any} index $j \in [m+1]$ such that $\S_{j}$ is distinct from $\S_{1}$ and $\S_{{m+1}}$. A key consequence is that the right side of \eqref{eqn:labeled_hard_mat} is clearly bounded by $O_d \left(d^{p_j}/n\right)$. Therefore the operator norm $\|\E M\|_{\rm op}$ tends to zero as long as there exists $\S_j$ that is distinct from $\S_1$ and $\S_{m+1}$ and such that $d^{p_j}=o(n)$.

\subsection{Related literature}
The study of Boolean matrix products and their spectral properties has been of long-standing interest in combinatorics, probability, and theoretical computer science. A central motivation arises from the analysis of Boolean functions, where Fourier-Walsh expansions provide a natural orthonormal basis for the $L_2$ space of functions on the hypercube. This approach has been explored in the works of O'Donnell~\cite{odonnell2014analysis} and de Wolf~\cite{de2008brief}, which establish Fourier analytic techniques for studying Boolean functions, including orthogonality properties and spectral concentration phenomena.

The study of random Boolean matrices has also been a core theme in random matrix theory. Classical results such as those by Vershynin~\cite{vershynin2018high}, Tropp \cite{tropp2015introduction}, and Tao and Vu~\cite{tao2012random} analyze the spectral properties of random matrices with independent rows, columns, and/or entries. In a similar spirit, the line of works \cite{huang2022matrix,baghal2020matrix,kathuria2020concentration,henriksen2020concentration} studies the concentration phenomenon for products of independent random matrices.
In contrast, products of Fourier-Walsh matrices that we investigate here exhibit nontrivial dependencies  among the factors. %
More recently, random Boolean matrices have appeared in applications such as learning theory and theoretical machine learning. Studies by Feldman et al.~\cite{feldman2017statistical} and Servedio and Tan~\cite{servedio2020attribute} have leveraged Fourier methods to understand the learnability of Boolean functions under noise. 
Finally, we note that our arguments are largely motivated by the \emph{moment method} from random matrix theory~\cite{wigner1955characteristic,erdos2013spectral,tao2010random}, which reduces moment calculations to combinatorial arguments involving paths in a graph.

\smallskip

The rest of the paper proceeds as follows. Section~\ref{sec:notat} records the basic notation we will use. Section~\ref{sec:prelim_results} establishes the key technical results that will be required to prove Theorem~\ref{lemma:feature_product_tech}. Finally, the complete proof of Theorem~\ref{lemma:feature_product_tech} appears in Section~\ref{sec:comp_proof}.

\section{Notation}\label{sec:notat}
We use the following notation throughout the paper. The symbols $\mathbb{N}$ and $\mathbb{R}$ will denote the sets of natural numbers and real numbers, respectively. Indices $i$ as subscripts on vectors $x_i$ will always denote the $i$'th coordinate while indices $i$ as superscripts $x^{(i)}$ will denote the $i$'th vector in a list. We will use $[m]$ as a shorthand for the set $\{1,2,\ldots,m\}$. The symbol $\mathbf{1}_m$ denotes the all-ones vector of length $m$, and $e_j$ denotes the $j$-th standard basis vector in $\R^m$. 
The $\ell_p$-norm of any vector $v\in \mathbb{R}^m$ will be denoted by $\|v\|_p = \left(\sum_i |v_i|^p\right)^{1/p}$. The symbols $\leq$ and ${\rm mod}$ will always be applied coordinate-wise to vectors $v\in \R^m$. For any matrix $A$, the submatrix $A_{\leq s, \leq q}$ consists of the first $s$ rows and first $q$ columns of $A$. Similar notation will be used for $A_{> s, \leq q}$. The cardinality of any set $S$ will be denoted by $|S|$. Finally, for a subset $S \subset [d]$, the symbol $x^S$ denotes the Fourier-Walsh polynomial $x^S=\prod_{i \in S} x_i$. Throughout, the hypercube $\mathbb{H}^d:=\{-1,1\}^d$ will be endowed with the uniform measure, and all expectations $\E[\cdot]$ will be taken with respect to this measure.

\section{Preliminary results}\label{sec:prelim_results}
In this section, we record key preliminary results that we will need to prove Theorem~\ref{lemma:feature_product_tech} in Section~\ref{sec:comp_proof}. The reader can safely skip this section during the first pass and refer back to it when reading Section~\ref{sec:comp_proof}.

We begin with the following lemma, which counts the number of binary matrices with even row sums and whose entries sum to a fixed constant.

\begin{lemma}[Binary matrices I]\label{lem:bin_matri}
    Fix two integers $p,q\in \mathbb{N}$. Then the set of binary matrices
    \begin{equation}\label{eqn:binary_mat}
    M_p=\left\{A\in \{0,1\}^{d\times q}: \sum_{i,j}A_{ij}=p~~\textrm{and}~~ A \cdot\mathbf{1}_{q}=0~\rm{mod}~2\right\},
    \end{equation}
    has cardinality bounded by $(q^2d)^{p/2}$.
\end{lemma}
\begin{proof}
Define $m_p=|M_p|$. We may assume that $p$ is even since otherwise $M_p$ is the emptyset and the result holds trivially.
Notice that every matrix in $M_{p}$ can be obtained from a matrix in $M_{p-2}$ by choosing a row in $[d]$ and replacing two zero entries in that row by ones. It follows that the values $m_p$ satisfy the recursion
$m_{p}\leq d\cdot {q\choose 2} \cdot m_{p-2}.$
Unrolling the recursion completes the proof.
\end{proof}

It will also be useful to count the number of binary matrices with prescribed column sums and rows sums modulo two. This is the content of the following corollary.

\begin{corollary}[Binary matrices II]\label{cor:mod2_cube}
    Fix an integer $q\in \mathbb{N}$, a vector $p\in \mathbb{N}^q$, and vector $v\in \{0,1\}^d$. Then the set of binary matrices
    \begin{equation}\label{eqn:binary_mat2}
    \left\{A\in \{0,1\}^{d\times q}: ~A \cdot{\mathbf{1}}_{q}=v~{\rm mod}~2~\textrm{and}~ A\tran \cdot {{\mathbf 1}}_{d}\leq p\right\},
    \end{equation}
    has cardinality bounded by $2^{q\|v\|_1+2}(q^2d)^{(\|p\|_1-\|v\|_1)/2}$.
\end{corollary}
\begin{proof}
To simplify notation, set $s=\|v\|_1$. Without loss of generality, we may permute the coordinates of $v$ in order to have the form $v=(\mathbf{1}_{s}, 0_{d-s})$. Let $A$ be any matrix contained in the set \eqref{eqn:binary_mat2}.
Clearly, the matrices contained in the set \eqref{eqn:binary_mat2} can have at most $2^{sq}$ possible submatrices $A_{\leq s,\leq q}$. Given such a submatrix $A_{\leq s,\leq q}$, the remaining submatrix $B=A_{>s,\leq q }$ lies in $\{0,1\}^{(d-s)\times q}$, the row sum $B\cdot\mathbf{1}_q$ is coordinate-wise even, and we have
$$\sum_{ij} B_{ij}=\sum_{j}\sum_{i}B_{ij}\leq\sum_{j} (p_j-\|A_{\leq s,\leq q}e_j\|_1)=\|p\|_1-\|A_{\leq s,\leq q}\|_1\leq \|p\|_1-\|v\|_1.$$
Appealing to Lemma~\ref{lem:bin_matri}, we see that the number of such matrices $B$ is bounded by $\sum_{k=0}^{\|p\|_1-\|v\|_1} |M_k|$. In the trivial setting $q= d= 1$, the cardinality of the set \eqref{eqn:binary_mat2} is clearly bounded by $2$. Now we consider the case when $d$ or $q$ is strictly larger than one. Then using the formula for the sum of a geometric series, we obtain
\begin{align*}
    \sum_{k=0}^{\norm{p}_1 - \norm{v}_1} |M_k| \leq \frac{(\sqrt{q^2d})^{\norm{p}_1 -\norm{v}_1 +1} - 1}{\sqrt{q^2d} - 1} \leq (\sqrt{q^2d})^{\norm{p}_1 -\norm{v}_1} \cdot \frac{\sqrt{q^2d}}{\sqrt{q^2d}-1}\leq 4(q^2d)^{(\|p\|_1-\|v\|_1)/2},
\end{align*}
thereby completing the proof.
\end{proof}

We now rephrase Corollary~\ref{cor:mod2_cube}  in terms of the Fourier-Walsh polynomials.

\begin{proposition}[Products of Fourier-Walsh monomials]\label{prop:basic_cpount_fourier}
Fix a Fourier-Walsh polynomial $x^{\tilde S}$ on $\{-1,1\}^d$ with a constant degree $|\tilde S|\leq \tilde p$ and fix some constants $q\in \mathbb{N}$ and $p\in \mathbb{N}^q$. Let $\mathcal{S}_1,\ldots, \mathcal{S}_q$ be families of sets such that each $\mathcal{S}_i$ consists of sets $S\subset [d]$ satisfying $|S|\leq p_i$. Then there are at most $O_d(d^{(\|p\|_1-\tilde p)/2})$ many tuples $(x^{S_1}, x^{S_2},\ldots,x^{S_q})$ with $S_i\in \mathcal{S}_i$ such that the polynomial $x^{S_1} x^{S_2}\cdots x^{S_q}$ coincides with $x^{\tilde S}$ on the hypercube.
\end{proposition}
\begin{proof}
For any set $S\subset[d]$, let $e_S\in\{0,1\}^d$ be a vector having ones along all coordinates in $S$ and zero otherwise. Then every tuple $(x^{S_1}, x^{S_2},\ldots,x^{S_q})$ with $S_i\in \S_i$ can be identified with a binary matrix $A\in \{0,1\}^{d\times q}$ having $e_{S_1},\ldots, e_{S_q}$ as its columns. Since each set $S\in \S_i$ satisfies $|S|\leq p_i$, we deduce $A\tran \cdot {{\mathbf 1}}_{d}\leq p$. Moreover, the product $x^{S_1} x^{S_2}\cdots x^{S_q}$ coincides with $x^{\tilde S}$ on the Hypercube precisely when equality $A \cdot{\mathbf{1}}_{q}=e_{\tilde S}~{\rm mod}~2$ holds. An application of Corollary~\ref{cor:mod2_cube} with $v=e_{\tilde S}$ completes the proof.
\end{proof}

Proposition \ref{prop:fourier_basis_product_card} below bounds an expectation of a sum of Fourier Walsh monomials over varying index sets. The key idea is to use Proposition \ref{prop:basic_cpount_fourier} in order to count the number of nonzero summands. 

\begin{proposition}[Sums of Fourier-Walsh polynomials]\label{prop:fourier_basis_product_card}
Fix some constants $q\in \mathbb{N}$ and $p\in \mathbb{N}^q$, and
fix a Fourier-Walsh polynomial $x^{S^*}$ on $\R^d$ with constant degree $p^*=|S^*|$.  Let $\mathcal{S}_1,\ldots, \mathcal{S}_q$ be collections of sets such that each collection $\mathcal{S}_i$ consists of sets $S\subset [d]$ satisfying $|S|\leq p_i$. Then the following holds:
\begin{align}\label{eq:card_bound_2}
     \sum_{S_1\in \S_1,S_2\in \S_2\ldots,S_q \in \S_q}\E\brac{x^{S_1}x^{S_2}\ldots x^{S_q} x^{S^*}} = O_{d}\round{d^{(\|p\|_1 - p^*)/2}}.
\end{align}
 \end{proposition}
\begin{proof}
Observe that the expectation $\E[x^{S_1} x^{S_2}\cdots x^{S_q}x^{S^*}]$ of each summand in \eqref{eq:card_bound_2} is one if $x^{S_1} x^{S_2}\cdots x^{S_q}$ coincides with $x^{S^*}$ on the hypercube and is zero otherwise. Corollary~\ref{prop:basic_cpount_fourier} in turn shows that 
the number of ways of writing $x^{S^*}$ as $x^{S_1} x^{S_2}\cdots x^{S_q}$ for some $S_1\in \S_1,S_2\in \S_2\ldots,S_q \in \S_q$, or equivalently the number of nonzero summands in \eqref{eq:card_bound_2},
 is $O_d(d^{(\|p\|_1-p^*)/2})$, thereby confirming \eqref{eq:card_bound_2}.
\end{proof}

We now extend Proposition \ref{prop:fourier_basis_product_card} to include {\em weighted} sums of Fourier-Walsh polynomials.  The following proposition provides several key inequalities that control the magnitude of these weighted sums in terms of the norms of the coefficient vectors. Additionally, we allow to enforce equalities between some of the sets $S_i$ that we sum over, which will be important in the next section. This proposition is the main result of the section.

\begin{proposition}[Weighted sums of Fourier-Walsh polynomials]\label{prop:fourier_basis_product_card_2}
Fix some some constants $q\in \mathbb{N}$ and $p\in \mathbb{N}^q$ and
fix a Fourier-Walsh polynomial $x^{S^*}$ on $\R^d$ with constant degree $p^*=|S^*|$. Let $\mathcal{S}_1,\ldots, \mathcal{S}_q$ be collections of sets such that each collection $\mathcal{S}_i$ consists of sets $S\subset [d]$ satisfying $|S|\leq p_i$.  Then for any vectors $a\in \R^{S_{q-1}}$ and $b\in \R^{S_{q}}$, the four estimates hold:
\begin{align}\label{eq:card_bound_3_group}
\sum_{S_1\in \S_1,S_2\in \S_2\ldots,S_q \in \S_q}b_{S_q}\E\brac{x^{S_{1}}x^{S_{2}}\ldots x^{S_{q}}} = \norm{b}_2\cdot O_{d}(d^{(\sum_{t=1}^{q-1} p_t)/2}),   
\end{align}

\begin{align}\label{eq:card_bound_5}
     \sum_{S_1\in \S_1,S_2\in \S_2\ldots,S_q \in \S_q}b_{S_q}\E\brac{x^{S_1}x^{S_2}\ldots x^{S_q}x^{S^*}} = \norm{b}_2\cdot O_{d}(d^{(\sum_{t=1}^{q-1} p_t )/2}),
\end{align}
\begin{align}\label{eq:card_bound_4}
     \sum_{S_1\in \S_1,S_2\in \S_2\ldots,S_q \in \S_q}a_{S_{q-1}}b_{S_q}\E\brac{x^{S_1}x^{S_2}\ldots x^{S_{q-1}}x^{S_q}} = \norm{a}_2 \norm{b}_2\cdot O_{d}(d^{(\sum_{t=1}^{q-2} p_t)/2}).
\end{align}
\begin{align}\label{eq:card_bound_6}
     \sum_{S_1\in \S_1,S_2\in \S_2\ldots,S_q \in \S_q}a_{S_{q-1}}b_{S_q}\E\brac{x^{S_1}x^{S_2}\ldots x^{S_{q-1}}x^{S_q}x^{S^*}} = \norm{a}_2 \norm{b}_2\cdot O_{d}(d^{(\sum_{t=1}^{q-2} p_t)/2}).
\end{align}

\end{proposition}

\begin{proof}
Without loss of generality, we may assume that $a_{S_q},b_{S_q}$ are nonnegative since the expectations in the sum are all nonnegative. 
We begin by verifying \eqref{eq:card_bound_3_group}.
Taking the summation over $S_q$ first, the left side of equation \eqref{eq:card_bound_3_group} is clearly bounded by:
    \begin{equation}\label{eqn:bound_interm_appp}
         \sum_{S_q \in \S_q}b_{S_q}\sum_{S_1\in \S_1,S_2\in \S_2\ldots, S_{q-1}\in \S_{q-1}}\E\brac{x^{S_1}x^{S_2}\ldots x^{S_q}}.
    \end{equation}
Invoking Eq.~(\ref{eq:card_bound_2}) with $S^*=S_q$, we immediately deduce that \eqref{eqn:bound_interm_appp} is bounded by 
\begin{align*}
      \sum_{S_q \in \S_q} b_{S_q} \cdot O_d(d^{(\sum_{t=1}^{q-1}p_t - p_q)/2})
    &\leq O_d(d^{(\sum_{t=1}^{q-1}p_t)/2})\sum_{S_q \in \S_q}b_{S_q} \cdot O_d(d^{-p_q/2})\\
    &=O_d(d^{(\sum_{t=1}^{q-1}p_t)/2})\cdot \|b\|_1\cdot O_d(|\mathcal{S}_q|^{-1/2})\\
    &= O_d(d^{(\sum_{t=1}^{q-1}p_t)/2})\cdot \norm{b}_2
\end{align*}
where the last inequality uses the equivalence of $\ell_1$ and $\ell_2$ norms.

Next, we verify Eq.~(\ref{eq:card_bound_5}). Similarly, we take the summation over $S_q$ first and  the left side of equation \eqref{eq:card_bound_5} becomes:
\begin{align}\label{eq:bound_5_interm}
     \sum_{S_q \in \S_q}b_{S_q}\sum_{S_1\in \S_1,S_2\in \S_2\ldots,S_{q-1} \in \S_{q-1}}\E\brac{x^{S_1}x^{S_2}\ldots x^{S_q}x^{S^*}}.
\end{align}
We now decompose the outer sum based on the size of the intersection $k=|S_q\cap S^*|$, which ranges from $k=1,\ldots,p_q \wedge p^*$. We thus obtain
\begin{align}
     &~~~\sum_{S_q \in \S_q}b_{S_q}\sum_{S_1\in \S_1,S_2\in \S_2\ldots,S_{q-1} \in \S_{q-1}}\E\brac{x^{S_1}x^{S_2}\ldots x^{S_q}x^{S^*}}\notag\\
     &= \sum_{k=0}^{p_q \wedge p^*}\sum_{S_q\in \S_q, |S_q\cap S^*| = k} b_{S_q} \sum_{S_1\in \S_1,S_2\in \S_2\ldots,S_{q-1} \in \S_{q-1}}\E\brac{x^{S_1}x^{S_2}\ldots x^{S_q}x^{S^*}}\notag\\
     &= \sum_{k=0}^{p_q \wedge p^*}\sum_{S_q\in \S_q, |S_q\cap S^*| = k} b_{S_q} \cdot O_d\round{d^{\sum_{t=1}^{q-1}p_t/2- (p_q+p^* - 2k)/2}}\label{eq:bound_5_interm_2}.
\end{align}
where the last inequality follows from equation~(\ref{eq:card_bound_2}).

Observe that the number of sets $S_q\in \S_q$ satisfying $|\S_q\cap S^*| = k$ is of the order $O_d(d^{p_q-k})$. Therefore Eq.~(\ref{eq:bound_5_interm_2}) can be further bounded by
\begin{align}
    &~~~O_d\round{d^{(\sum_{t=1}^{q-1}p_t)/2}}\cdot\sum_{k=0}^{p_q \wedge p^*}\sum_{S_q\in \S_q, |S_q\cap S^*| = k} b_{S_q} \cdot O_d\round{d^{- (p_q+p^* - 2k)/2}}\notag \\
    &=O_d\round{d^{(\sum_{t=1}^{q-1}p_t)/2}} \sum_{k=0}^{p_q \wedge p^*} \sqrt{\sum_{S_q\in \S_q, |S_q\cap S^*| = k} b_{S_q}^2}\cdot O_d(d^{-(p^*-k)/2})\label{eqn:CS_ineqblue}\\
    &= O_d\round{d^{(\sum_{t=1}^{q-1}p_t)/2}} \sum_{k=0}^{p_q \wedge p^*} \sqrt{\sum_{S_q\in \S_q, |S_q\cap S^*| = k} b_{S_q}^2}\cdot O_d(1)\notag\\
    &= O_d\round{d^{(\sum_{t=1}^{q-1}p_t)/2}} \cdot{\norm{b}_2},\label{eqn:need_expr_norm}
\end{align}
where \eqref{eqn:CS_ineqblue} follows from the equivalence of $\ell_1$ and $\ell_2$ norms.

Next, we verify~\eqref{eq:card_bound_4}. To this end, we will break up the sum \eqref{eq:card_bound_4} by looking at the size of the intersection $|S_{q}\cap S_{q-1}|$, which clearly ranges from zero to $p_q\wedge p_{q-1}$. With this in mind, the left side of \eqref{eq:card_bound_4} can be equivalently written as
\begin{align}\label{eq:bound_a_b_1}
&=    \sum_{k=0}^{p_q\wedge p_{q-1}}\sum_{\substack{S_{q-1}\in \S_{q-1},S_{q}\in \S_{q}\\|S_{q}\cap S_{q-1}| = k }}a_{S_{q-1}}   b_{S_q}  \sum_{S_1\in \S_1,S_2\in \S_2\ldots, S_{q-2}\in \S_{q-2}}\E\brac{x^{S_1}x^{S_2}\ldots x^{S_{q-1}}x^{S_q}}.
\end{align}
Fix now $S_{q-1}\in \S_{q-1}$ and $S_{q}\in \S_{q}$ satisfying $|S_{q}\cap S_{q-1}| = k$. 
Observe that the polynomial $x^{S_{q-1}}x^{S_q}$ coincides on the Hypercube with $x^{S^*}$ where the set $S^*:=S_{q-1}\cup S_q\setminus (S_{q-1}\cap S_q)$ has cardinality $|S^*|=p_{q-1}+p_q-2k$. 
Thus, invoking Eq.~(\ref{eq:card_bound_2}) yields
\begin{align*}
    \sum_{S_1\in \S_1,S_2\in \S_2\ldots, S_{q-2}\in \S_{q-2}}\E\brac{x^{S_1}x^{S_2}\ldots x^{S_{q-1}}x^{S_q}} = O_d(d^{(\sum_{t=1}^{q-2}p_t - p_{q-1} - p_q)/2 + k}).
\end{align*}
Plugging this expression into Eq.~\eqref{eq:bound_a_b_1}, shows that \eqref{eq:bound_a_b_1} is upper-bounded by 
\begin{align}\label{eq:bound_a_b_2}
    \sum_{k=0}^{p_q\wedge p_{q-1}} O_d(d^{(\sum_{t=1}^{q-2}p_t - p_{q-1} - p_q)/2 + k}) \cdot \sum_{\substack{S_{q-1}\in \S_{q-1},S_{q}\in \S_{q}\\|S_q\cap S_{q-1}|=k}}a_{S_{q-1}}   b_{S_q}.
\end{align}

It remains  to upper-bound the inner-sum in \eqref{eq:bound_a_b_2}. To this end, define an undirected bipartite graph $G = (V_1\cup V_2,E)$ where $V_1 = \S_{q-1}$ and $V_2 = \S_{q}$ are disjoint sets of vertices and an 
 edge $(S_{q-1},S_q) \in E$ is present if the equation $|S_q\cap S_{q-1}|=k$ holds. We further define a bidjacency matrix $A$  of $G$ by 
\begin{align}\label{eq:a_b_bidjacency}
    A_{S_{q-1},S_q} = \begin{cases}
        &1,~~~{\rm if}~(S_{q-1},S_q) \in E,\\
        &0,~~~{\rm otherwise.}
    \end{cases}
\end{align}
Then, clearly we can write
\begin{equation}\label{eq:bound_a_b_3}
    \sum_{\substack{S_{q-1}\in \S_{q-1},S_{q}\in \S_{q}\\{|S_q\cap S_{q-1}|=k}}}a_{S_{q-1}}   b_{S_q} =  a^\top Ab\leq \|A\|_{\rm op} \|a\|_2\|b\|_2.
\end{equation}
Note that the maximal degree of every vertex in $V_1$ is of the order $O_d(d^{p_q - k})$ and in $V_2$ is  of the order  $O_d(d^{p_{q-1} - k})$. In this case, the matrix $A$  satisfies $\norm{A}_1 = O_d(d^{p_{q-1} - k})$ and $\norm{A}_\infty =O_d(d^{p_q - k})$, where $\norm{A}_1$ and $\norm{A}_\infty$ denote the maximum column norm and maximum row norm, respectively. We therefore deduce
\begin{align}\label{eq:a_b_norm_A}
    \snorm{A}\leq \sqrt{\norm{A}_1\norm{A}_\infty } =O_d(d^{(p_{q-1}+p_q)/2 - k}),
\end{align}
where the first inequality is valid for any matrix. Therefore, the right side of \eqref{eq:bound_a_b_3} is bounded by 
$ O_d(d^{(p_{q-1}+p_q)/2 - k})\cdot\norm{a}_2\norm{b}_2.$
Plugging this back into Eq.~(\ref{eq:bound_a_b_2}) completes the proof.

Lastly, we verify (\ref{eq:card_bound_6}). Similar to Eq.~(\ref{eq:bound_a_b_1}), we break up the sum by looking at the size of the intersection $S_q\cap S_{q-1}$ and $(S_q \Delta S_{q-1})\cap S^*$, where $S_q \Delta S_{q-1} := [(S_q \cup S_{q-1}) \slash (S_{q}\cap S_{q-1})]$ denotes the symmetric difference between $S_q$ and $S_{q-1}$. Then the left side of Eq.~(\ref{eq:card_bound_6}) can be written as
\begin{align}\label{eq:bound_a_b_4}
\sum_{k=0}^{p_q\wedge p_{q-1}}\sum^{(p_{q} + p_{q-1}-2k) \wedge p^*}_{j=0} \sum_{\substack{S_{q-1}\in \S_{q-1},S_{q}\in \S_{q}\\|S_{q}\cap S_{q-1}| = k  \\ \abs{(S_q \Delta  S_{q-1})\cap S^*} = j}}a_{S_{q-1}}   b_{S_q} \round{\sum_{S_1\in \S_1,\ldots, S_{q-2}\in \S_{q-2}}\E\brac{x^{S_1}x^{S_2}\ldots x^{S_{q-1}}x^{S_q}x^{S^*}}}.
\end{align}
We fix $S_{q-1} \in \S_{q-1}$ and $S_{q}\in \S_q$ satisfying $|S_{q}\cap S_{q-1}| = k$ and $\abs{(S_q \Delta  S_{q-1})\cap S^*} = j$. We treat the product $x^{S_{q-1}}x^{S_q}x^{S^*}$ as the monomial $x^{S^*}$ in Eq.~(\ref{eq:card_bound_2}), which has cadinality at least $p_{q}+p_{q-1}+ p^*- 2k - 2j$. Thus, invoking Eq.~(\ref{eq:card_bound_2}) yields
\begin{align*}
    \sum_{S_1\in \S_1,\ldots, S_{q-2}\in \S_{q-2}}\E\brac{x^{S_1}x^{S_2}\ldots x^{S_{q-1}}x^{S_q}x^{S^*}} =  O_d(d^{(\sum_{t=1}^{q-2}p_t - p_{q-1} - p_q - p^*)/2 + k + j}).
\end{align*}
Plugging this expression into Eq.~\eqref{eq:bound_a_b_4} shows that \eqref{eq:bound_a_b_4} is upper-bounded by
\begin{align}\label{eq:bound_a_b_5}
    \sum_{k=0}^{p_q\wedge p_{q-1}}\sum^{(p_{q} + p_{q-1}-2k) \wedge p^*}_{j=0}  O_d(d^{(\sum_{t=1}^{q-2}p_t - p_{q-1} - p_q - p^*)/2 + k + j}) \cdot \sum_{\substack{S_{q-1}\in \S_{q-1},S_{q}\in \S_{q}\\|S_{q}\cap S_{q-1}| = k  \\ \abs{(S_q \Delta  S_{q-1})\cap S^*} = j}}a_{S_{q-1}}   b_{S_q}.
\end{align}

It remains  to upper-bound the inner-sum in~\eqref{eq:bound_a_b_5}. To this end, we split $(S_q \Delta  S_{q-1})\cap S^*$ as a disjoint union of $[S_{q-1} \slash (S_{q}\cap S_{q-1})]\cap S^*$ and $[S_{q} \slash (S_{q}\cap S_{q-1})]\cap S^*$. Then we write the inner-sum in~\eqref{eq:bound_a_b_5} as
\begin{align}\label{eq:bound_a_b_6}
    \sum_{\substack{S_{q-1}\in \S_{q-1},S_{q}\in \S_{q}\\|S_{q}\cap S_{q-1}| = k  \\ \abs{(S_q \Delta  S_{q-1})\cap S^*} = j}}a_{S_{q-1}}   b_{S_q} = \sum_{\substack{j_1,j_2 : j_1+j_2 = j }}\sum_{\substack{S_{q-1}\in \S_{q-1},S_{q}\in \S_{q}\\|S_{q}\cap S_{q-1}| = k \\  \abs{[S_{q-1} \slash (S_{q}\cap S_{q-1})]\cap S^*} = j_1  \\ \abs{[S_{q} \slash (S_{q}\cap S_{q-1})]\cap S^*} = j_2}}a_{S_{q-1}}   b_{S_q}.
\end{align}
Now we fix $j_1$ and $j_2$. Similar to the proof of \eqref{eq:card_bound_4}, we define an undirected bipartite graph $G = (V_1\cup V_2,E)$ where $V_1 = \S_{q-1}$ and $V_2 = \S_{q}$ and an 
 edge $(S_{q-1},S_q) \in E$ is present if the equations $|S_q\cap S_{q-1}|=k$, $\abs{[S_{q-1} \slash (S_{q}\cap S_{q-1})]\cap S^*} = j_1$ and $\abs{[S_{q} \slash (S_{q}\cap S_{q-1})]\cap S^*} = j_2$ hold. Accordingly, we define a bidjacency matrix $A$ of $G$ as in Eq.~(\ref{eq:a_b_bidjacency}).

Note that the maximal degree of every vertex in $V_1$ is of the order $O_d(d^{p_q - k - j_2})$ and in $V_2$ is  of the order  $O_d(d^{p_{q-1} - k- j_1})$. The same argument as the one establishing Eq.~(\ref{eq:a_b_norm_A}) therefore implies 
 \begin{align*}
     \snorm{A}\leq \sqrt{\norm{A}_1\norm{A}_\infty } =O_d(d^{(p_{q-1}+p_q - j)/2 - k}).
 \end{align*}
 Then for each $j_1$ and $j_2$, we deduce
\begin{align*}
\sum_{\substack{S_{q-1}\in \S_{q-1},S_{q}\in \S_{q}\\|S_{q}\cap S_{q-1}| = k \\  \abs{[S_{q-1} \slash (S_{q}\cap S_{q-1})]\cap S^*} = j_1  \\ \abs{[S_{q} \slash (S_{q}\cap S_{q-1})]\cap S^*} = j_2}}a_{S_{q-1}}   b_{S_q} = a\tran A b \leq  \snorm{A} \|a\|_2\|b\|_2 = \|a\|_2\|b\|_2\cdot O_d(d^{(p_{q-1}+p_q - j)/2 - k}).
\end{align*}

Since $j$ is of the order $O_d(1)$, plugging the above bound into Eq.~\eqref{eq:bound_a_b_6} gives
\begin{align*}
      \sum_{\substack{S_{q-1}\in \S_{q-1},S_{q}\in \S_{q}\\|S_{q}\cap S_{q-1}| = k  \\ \abs{(S_q \Delta  S_{q-1})\cap S^*} = j}}a_{S_{q-1}}   b_{S_q} = \|a\|_2\|b\|_2\cdot O_d(d^{(p_{q-1}+p_q - j)/2 - k}).
\end{align*}
Subsequently,  plugging this estimate into Eq.~\eqref{eq:bound_a_b_5} yields an upper bound for Eq.~\eqref{eq:bound_a_b_4}:
\begin{align*}
      &~~~\sum_{k=0}^{p_q\wedge p_{q-1}}\sum^{(p_{q} + p_{q-1}-2k) \wedge p^*}_{j=0}  O_d(d^{(\sum_{t=1}^{q-2}p_t - p_{q-1} - p_q - p^*)/2 + k + j}) \cdot \|a\|_2\|b\|_2\cdot O_d(d^{(p_{q-1}+p_q - j)/2 - k}) \\
      &=\sum_{k=0}^{p_q\wedge p_{q-1}}\sum^{(p_{q} + p_{q-1}-2k) \wedge p^*}_{j=0}  \norm{a}_2\norm{b}_2\cdot O_d(d^{(\sum_{t=1}^{q-2}p_t)/2})\\
      &=\norm{a}_2\norm{b}_2\cdot O_d(d^{(\sum_{t=1}^{q-2}p_t)/2})
\end{align*}
where the first equality follows from the fact 
that $p^* \geq j$ holds and the second equality follows from the fact that $j$ and $k$ scale as $O_d(1)$. The proof is complete.
\end{proof}

\begin{remark}\label{rem:needed_remark}{\rm 
In Section~\ref{sec:comp_proof}, we will consider sums of the form \eqref{eq:card_bound_3_group}-\eqref{eq:card_bound_6} but where the tuples $(S_1,\ldots, S_q)$ may be further constrained to lie in some set $\Gamma\subseteq \S_1\times \ldots\S_q$. Then all the estimates \eqref{eq:card_bound_3_group}-\eqref{eq:card_bound_6} remain true with the sum $\sum_{S_1\in \S_1,S_2\in \S_2\ldots,S_q \in \S_q}$ replaced by  $\sum_{(S_1,\ldots,S_q) \in \Gamma}$. To see this, simply note that since each expectation in the sum is nonnegative we may assume that $a_{S_{q-1}}$ and $b_{S_{q}}$ are nonnegative, and therefore any upper bound on  $\sum_{S_1\in \S_1,S_2\in \S_2\ldots,S_q \in \S_q} [\cdots]$ is a valid upper bound on the sum $\sum_{(S_1,\ldots,S_q) \in \Gamma}[\cdots]$.}
\end{remark}

\section{Proof of Theorem~\ref{lemma:feature_product_tech}}\label{sec:comp_proof}
We now have all the ingredients to prove Theorem~\ref{lemma:feature_product_tech}, the main result of the paper. 
Given any unit vectors $b_1\in \R^{\S_1}$ and $b_2\in \R^{\S_{m+1}}$, we will show
\begin{align*}
   b_1\tran \E \brac{ M}  b_2\leq  O_d \left(d^{p_j}\right)\|w^{(1)}\|_\infty \|w^{(m+1)}\|_\infty,
\end{align*}
which directly implies the estimate in Eq.~\eqref{eqn:labeled_hard_mat}.

We begin with some notation. For any $j\in [m+1]$ and any set $k\in \S_j$, we let $a^j_k$ denote the column of the matrix $A_j$ indexed by $k$. For each integer $j$, we let $a^j_{i,k}$ denote the $i$'th entry of $a^j_k$. To simplify notation, we will use the symbols $a^j_{i,k}$ and $a^j_{k,i}$ interchangeably. 
Noting the equality $A_j A_j^\top=\sum_{k=1}^{|\S_j|} a^j_k(a^j_k)^\top $ for each $j\in [p+1]$, we may expand the pairwise products of matrices $A_jA_j^\top$ yielding:  
     \begin{align}
        &\mathbf{S}:=b_1^\top A_1^\top \Bigl(A_2 A_2\tran \Bigr) \cdots \Bigl(A_{m} A_{m}\tran \Bigr) A_{m+1} \, b_2 \notag\\
&~~~~=\sum_{\substack{i_1,i_2,\ldots,i_{m} \in[n],\\ k_1\in \S_{1},k_2\in \S_{2},\ldots,k_{m+1}\in \S_{{m+1}}}} (b_1)_{k_1}a^1_{k_1,i_1} a^2_{i_1,k_2}  a^2_{k_2,i_2} a^3_{i_2,k_3} \ldots a^{m}_{k_{m},i_{m}} a^{m+1}_{i_{m},k_{m+1}} (b_2)_{k_{m+1}}\notag\\
        &=\sum_{\substack{i_1,i_2,\ldots,i_{m} \in[n],\\  k_1\in \S_{1},k_2\in \S_{2},\ldots,k_{m+1}\in \S_{m+1}}}  (b_1)_{k_1} (b_2)_{k_{m+1}}\prod_{j=1}^{m} a_{k_j,i_j}^j a_{i_j,k_{j+1}}^{j+1}.\label{eqn:key_path_expr}
     \end{align}

\paragraph{Partition of indices.}Now, let $\Pi_{k}$ denote the set of all partitions of the set $[k]$. Thus a partition $\pi\in \Pi_{k}$ has the form $\pi=\{T_1,\ldots,T_q\}$ for some blocks $T_i\subset [k]$ that are nonempty, disjoint, and so that the  union $\cup_{j=1}^q T_j$ equals $[k]$. For any $r\in [k]$, we let $\pi(r)$ denote the unique index $i$ of the block $T_i$ satisfying $r\in T_i$.

Notice that the elements of the set 
\begin{equation}\label{eqn:prod_ints}
(i_1,\ldots,i_{m}) \in \underbrace{[n]\times\ldots \times [n]}_{m ~{\rm times}}
\end{equation}
can be grouped based on identifying the indecies $j$ and $k$ such that $i_j$ and $i_k$ are equal.
 In other words, the set  \eqref{eqn:prod_ints} can be stratified by partitions $\pi_1\in\Pi_{m}$ so that equalities $i_r=i_t$ hold if and only if $\pi_1(r)=\pi_1(t)$. Similarly, the features $\S_1\times\ldots \S_{m+1}$ can be stratified by partitions $\pi_2\in\Pi_{m+1}$ so that equalities $k_r=k_t$ hold if and only if $\pi_2(r)=\pi_2(t)$. Since there are $O_d(1)$ many such partitions, we may fix a partition $\pi_1=\{T_1,\ldots, T_q\}\in\Pi_{m}$ and $\pi_2=\{G_1,\ldots,G_s\}\in\Pi_{m+1}$ and only focus on summands in \eqref{eqn:key_path_expr} satisfying 
\begin{align*}
&i_r=i_t~\textrm{if and only if}~~\pi_1(r)=\pi_1(t),\qquad \forall r,t\in [m],\\
&k_r=k_t~ \textrm{if and only if}~~\pi_2(r)=\pi_2(t)\qquad \forall r,t\in [m+1].
\end{align*}
We will denote the sum in \eqref{eqn:key_path_expr} indexed by the partitions $\pi_1$ and $\pi_2$ as $\mathbf{S}_{\pi}$.

\paragraph{Feature collapse.}Next, we simplify $\mathbf{S}_\pi$ by ``refining'' the partitions $\pi_1$ and $\pi_2$ by ``collapsing'' certain adjacent terms which results in deletions of indices from the partitions $\pi_1$ and $\pi_2$. To motivate the deletion process, observe that the equality holds:
\begin{align}\label{eqn:collapse_1}
   a^r_{i_{r-1},k_r}\cdot a^r_{k_r,i_r}=(w^{(r)}_{k_r})^2\qquad \textrm{whenever}~~ \pi_1(r-1)=\pi_1(r). 
\end{align}
Similarly, we have
\begin{align}\label{eqn:collapse_2}
   a^r_{k_r,i_r}\cdot a^{r+1}_{i_{r},k_{r+1}}=(w^{(r)}_{k_r})(w^{(r+1)}_{k_{r}})\qquad \textrm{whenever}~~ \pi_2(r)=\pi_2(r+1), 
\end{align}
where we use the fact that $A_r$ and $A_{r+1}$ are indexed by the same features whenever $\pi_2(r)=\pi_2(r+1)$.

Suppose now that \eqref{eqn:collapse_1} holds for some index $r$. Then we may replace the product 
$a^r_{k_r,i_r}\cdot a^{r+1}_{i_{r},k_{r+1}}$ by the constant $(w^{(r)}_{k_r})^2$. Thus, this ``feature collapse'' 
removes $k_r$ from the product leaving only the constant $(w^{(r)}_{k_r})^2$ in the summand. Thus, we may remove $r$ from the block $G\in \pi_2$ containing $r$.
Moreover, after the collapse we will be left with the product 
$a^{r-1}_{k_{r-1},i_{r-1}}a^{r+1}_{i_{r},k_{r+1}}$. 
Since equality  $\pi_1(r-1)=\pi_1(r)$ holds and $i_{r-1}$ and $i_{r}$ only appear in this term, we may merge the indices $i_{r-1}$ and $i_r$ into into a single index, say $i_{r-1}$. %
Exactly parallel reasoning applies to indices satisfying \eqref{eqn:collapse_2}, leading to ``sample collapse''. We may now iteratively collapse sample indices $r\in [m]$ and feature indices $[m+1]$ until no further collapse is possible.

The collapsing procedure is best illustrated pictorially by identifying the sum \eqref{eqn:key_path_expr} with the diagram
$$b_{1}\to k_1\xrightarrow{1} i_1\xrightarrow{2} k_2\xrightarrow{2} i_2\to\ldots k_{m}\xrightarrow{m} i_{m} \xrightarrow{m+1} k_{m+1}\to b_{2},$$
which alternates between feature indices and samples and the labels $r$ above the arrows indicate the matrix $A_r$ corresponding to the product.
Pictorially, \eqref{eqn:collapse_1} means that in the setting $\pi_1(r-1)=\pi_1(r)$ we can collapse subpaths of the form
\begin{align}\label{diag:collapse_1}
   k_{r-1}\xrightarrow{r-1} \boxed{i_{r-1}\xrightarrow{r} k_{r}\xrightarrow{r} i_r}\xrightarrow{r+1} k_{r+1}  \qquad   \textrm{to a subpath} \qquad k_{r-1}\xrightarrow{r-1} i_{r-1}  \xrightarrow{r+1} k_{r+1}, 
\end{align}
at a multiplicative cost of $(w^{(r)}_{k_r})^2$. Note that although $k_r$ no longer appears in the path, we still need to sum $(w^{(r)}_{k_r})^2$ over $k_r$. Similarly, \eqref{eqn:collapse_2} means that in the setting $\pi_2(r)=\pi_2(r+1)$ we can collapse subpaths of the form 
\begin{align}\label{diag:collapse_2}
  i_{r-1}\xrightarrow{r} \boxed{k_{r}\xrightarrow{r} i_r\xrightarrow{r+1} k_{r+1}}\xrightarrow{r+1} i_{r+1}  \qquad   \textrm{to a subpath} \qquad i_{r-1}\xrightarrow{r+1} k_{r+1}  \xrightarrow{r+1} i_{r+1},  
\end{align}
at a multiplicative cost of $(w^{(r)}_{k_r})^2$.
\footnote{Note that we replace  $a_{i_{r+1},k_{r}}^{r}$ by $a_{i_{r+1},k_{r+1}}^{r+1}$ when merging indices $r$ and $r+1$, which leads to an extra multiplicative cost $w^{(r)}_{k_{r}} / w^{(r+1)}_{k_{r}} $ aside from the multiplicative cost in Eq.~(\ref{eqn:collapse_2}). Multiplying the two costs together yields  $(w^{(r)}_{k_r})^2$. }
The iterative collapsing procedure then amounts to iteratively collapsing subpaths in such diagrams. Notice that each type of collapse removes a single pair of sample and feature indices $(i_r,k_r)$ from the product. Note, howevever, that we still need to sum over these indices the multiplicative costs $(w^{(r)}_{k_r})^2$.

\paragraph{Subscripts relabeling.} We relabel the subscripts of $i$ and $k$ to ensure their contuguity after each collapse. This can be achieved by relabeling $k_{r'}$ and $i_{r'}$ to be $k_{r'-1}$ and $i_{r'-1}$ respectively for all $ r'>r$. Then the subpaths in \eqref{diag:collapse_1} and \eqref{diag:collapse_2} become 
\begin{align}
    k_{r-1}\xrightarrow{r-1} i_{r-1}  \xrightarrow{r} k_{r},~~ {\rm and}~~~i_{r-1}\xrightarrow{r} k_{r}  \xrightarrow{r} i_{r},
\end{align}
and we may iterate the collapsing procedure until no further collapse is possible.

We will need to track the indices that have been deleted during the collapsing procedure. To this end, we let the sets $\L_1$ and $\L_2$ consist of the indices $r$ that have been removed due to \eqref{diag:collapse_1} and \eqref{diag:collapse_2}, respectively, during the iterative collapse process (in the original labeling). We define $L := |\L_1| +|\L_2|$. The relabeling naturally gives a bijection from  the new relabeled subscripts to
the original subscripts:
 $$\sigma\colon  [p+1-L] \to  [m+1]\slash (\L_1 \sqcup \L_2)$$

In order to simplify notation, define the constant $\emp$ by the expression
$\emp=p-L$. Recall that all the subscripts $r$ have been permuted by $\sigma$. In particular, all remaining indices now appear in $\mathcal{I}_{+}:=[\emp+1]$. To simplify notation, let us now relabel $\L_1$ and $\L_2$ so that they lie in $\mathcal{I}_{-}:=\{m+2-L,\ldots, m+1\}$ in an arbitrary order. With this notation, $\mathbf{S}_{\pi}$ has the explicit form:

\begin{align}\label{eq:relabel_product}
 \mathbf{S}_{\pi}=  \sum_{\substack{i_1,i_2,\ldots,i_{m},\\  k_1,k_2,\ldots,k_{m+1}}} \left(\prod_{r\in \mathcal{I}_-} (w^{(r)}_{k_r})^2 \right)(b_1)_{k_1}a^1_{k_1,i_1} a^2_{i_1,k_2}  a^2_{k_2,i_2} a^3_{i_2,k_3} \ldots a^{\emp}_{k_{\emp},i_{\emp}} a^{\emp+1}_{i_{\emp},k_{\emp+1}} (b_2)_{k_{\emp+1}}. 
\end{align}
Note that the sum in \eqref{eq:relabel_product}  is still taken over the partitions $\pi_1$ and $\pi_2$, which forces equalities among some of the values in $\{i_r\}_r$ and in $\{k_r\}_r$. Note that if there exists a block $P\in\pi_1$ such that $P\cap [\emp]$ is a singleton, then the expectation of $\mathbf{S}_{\pi}$ is zero. Therefore without loss of generality we may suppose $|T_i\cap [\emp]|$ is either zero or at least two for all $i\in [q]$.
Note moreover that the expectation of the summands in \eqref{eq:relabel_product} is independent of the indices $i_1,i_2,\ldots,i_{m}$. Therefore, we may suppose for the rest of the proof that $i_1,i_2,\ldots,i_{m}$ are fixed, bound the expectation of the sum \eqref{eq:relabel_product} over 
 varying indices $k_1,k_2,\ldots,k_{m+1}$, and multiply the bound by $n^{q}$.

We now split $\pi_2$ based on whether the block contains any of the noncollapsed indices:
\begin{align*}
E_1&:=\{u~\vert~G_u \cap \mathcal{I}_+ \neq \emptyset\}\qquad \textrm{and}\qquad
E_2:=\{u~\vert~G_u \cap \mathcal{I}_+ = \emptyset\}.
\end{align*}
We correspondingly denote all the indices in $E_1$ and $E_2$ by $K_1$ and $K_2$:
\begin{align*}
    K_1 := \bigsqcup_{u\in E_1}G_u,\qquad K_2 := \bigsqcup_{u\in E_2} G_u.
\end{align*}

\begin{figure}[htb!]
    \centering
    \vspace{-3cm}
    \includegraphics[width=0.8\linewidth]{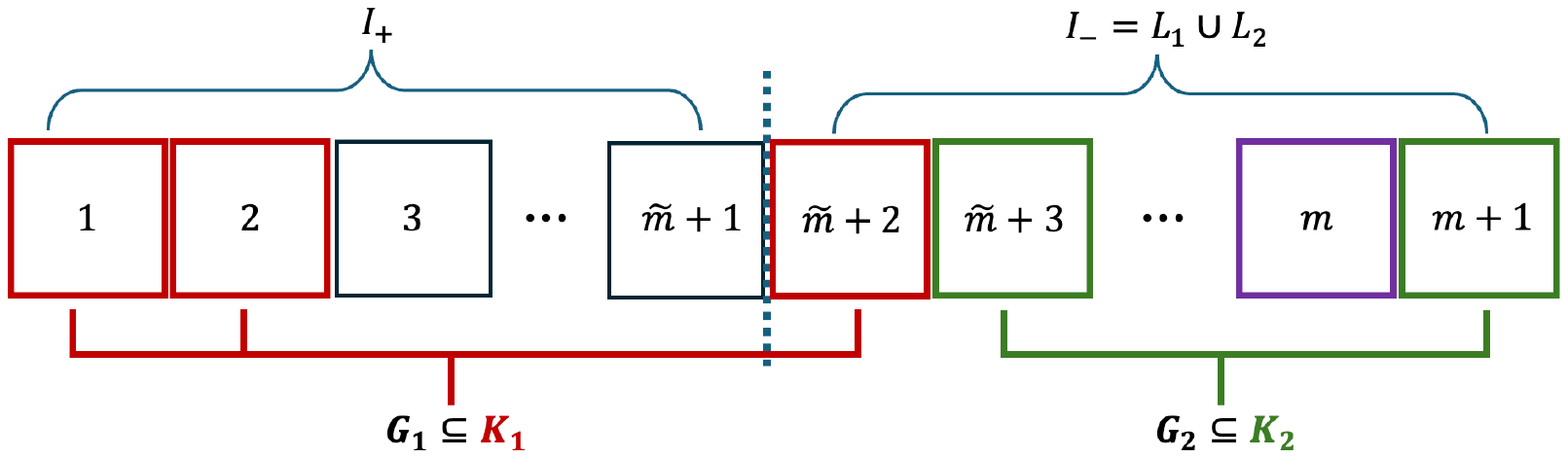}
    \vspace{-3cm}
    \caption{An illustration of $K$.}
    \label{fig:fig_1}
\end{figure}

Note that we have $K_1\cup K_2=[p+1]$ and moreover $\pi_2(k)$ and $\pi_2(k')$ are disjoint for any $k\in K_1$ and $k'\in K_2$.
Next, for each block index $u=1,\ldots, s$ of $\pi_2$ we define

\begin{align}\label{eqn:weights_feature}
    P_u &:= \prod_{r \in G_u \cap \mathcal{I}_- } (w^{(r)}_{k_r})^{\frac{2}{|G_u\cap \mathcal{I}_+|}}.
\end{align}
The powers on $(w^{(r)}_{k_r})^2$ arise because we will distribute $P_u$  uniformly across all elements in $G_u\cap \mathcal{I}_+$.
\begin{figure}[htb!]
    \centering
    \vspace{-3cm}
    \includegraphics[width=0.8\linewidth]{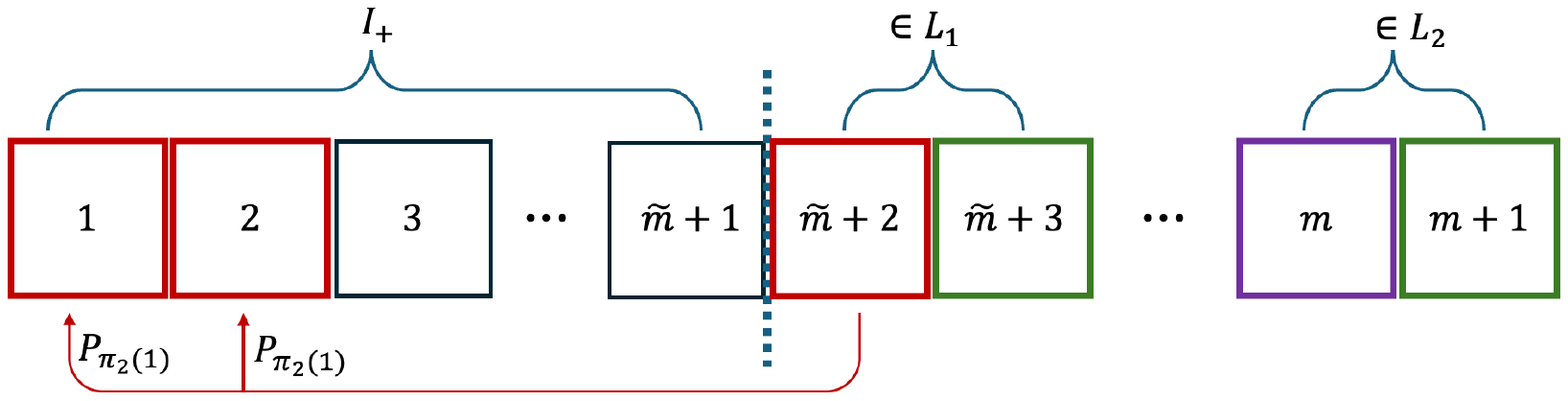}
    \vspace{-3cm}
    \caption{An illustration of weighted feature.}
    \label{fig:fig_2}
\end{figure}

Now we are ready to bound $\mathbf{S}_{\pi}$. In the following, we assume that not all features $k_r$ have collapsed to a single feature $k_1$. We will discuss this simple case at the end of the proof.

Let's first fix $i_1,i_2,\ldots,i_{m}$ and $k_1,k_2,\ldots,k_{m+1}$ and decompose the inner product
\begin{align*}
\prod_{r\in \mathcal{I}_-} (w^{(r)}_{k_r})^2=\left(\prod_{r  \in K_1 \cap \mathcal{I}_-} (w^{(r)}_{k_r})^2\right)\left(\prod_{r \in K_2 \cap \mathcal{I}_-} (w^{(r)}_{k_r})^2\right)
\end{align*}
Note for any $r\in K_2\cap \mathcal{I}_{-}$ the block $G_{\pi_2(r)}$ is contained fully in $\mathcal{I}_{-}$. Therefore, all the indices $k_r$ in this product are distinct from those appearing as a subscript on $a$'s. Using the fact that  $\pi_2(k)$ and $\pi_2(k')$ are disjoint for any $k\in K_1$ and $k'\in K_2$, we may sum over $k$ by summing independently over $K_1$ and $K_2$ thereby yielding:
\begin{align}
&\sum_{k_1,k_2,\ldots,k_{m+1}} \left(\prod_{r\in \mathcal{I}_-} (w^{(r)}_{k_r})^2 \right)\cdot \left[\cdots\right]\notag\\
&=\sum_{k_r:\, r\in K_1}\sum_{k_r:\, r\in K_2} \left(\prod_{r  \in K_1 \cap \mathcal{I}_-} (w^{(r)}_{k_r})^2\right)\left(\prod_{r \in K_2 \cap \mathcal{I}_-} (w^{(r)}_{k_r})^2\right) \cdot  \left[\cdots\right]\notag \\
&=\left(\sum_{k_r:\, r\in K_1} \left(\prod_{r \in K_1 \cap \mathcal{I}_-} (w^{(r)}_{k_r})^2\right) \left[\cdots\right]\right)\cdot\left(\sum_{k_r:\, r\in K_2} \left(\prod_{r  \in K_2 \cap \mathcal{I}_-}\right) (w^{(r)}_{k_r})^2\right)\label{eqn:ewqn_messproduct}
\end{align}
where $[\cdots]$ denotes the product $b_{c_1,k_1}a^1_{k_1,i_1}  \cdots a^{\emp+1}_{i_{\emp},k_{\emp+1}} b_{k_{\emp+1},c_{\emp+1}}$. Next, we will need the following claim.

\begin{claim}\label{claim:g_u_l_1}
For any $u\in E_2$, the intersection $G_u\cap \mathcal{L}_1$ is nonempty. 
\end{claim}
\begin{proof}
  Given $G_u$, we consider the last index $r\in \L_2 \cap G_u$ that was removed due to condition \eqref{eqn:collapse_2}.  There exists $r' \in \I^+\cup \L_1$ such that $\pi_2(r') = \pi_2(r)$ by definition. 
   Since each $G_u$ with $u\in E_2$ consists of subscripts from $\L_1 \cup \L_2$, we have $r'\in \L_1$ hence completing the proof. 
\end{proof}

Taking into account this claim, we see that for each $r\in K_2$, the set $G_{\pi_2(r)}$ intersects $\mathcal{L}_1$. Therefore, we may replace  $\sum_{k_r:\, r\in K_2}$ in \eqref{eqn:ewqn_messproduct} with $\sum_{k_r:\, r\in K_2\cap \mathcal{L}_1}$.
Splitting $K_2 \cap \I_-$ into $K_2 \cap \L_1$ and $K_2\cap \L_2$, the second summation in Eq.~\eqref{eqn:ewqn_messproduct} can be written as
\begin{align}\label{eqn:weight_assign_1}
    \sum_{k_r:\, r\in K_2} \prod_{r  \in K_2 \cap \mathcal{I}_-} (w^{(r)}_{k_r})^2 = \sum_{k_r:\, r\in K_2\cap \L_1} \round{\prod_{r  \in K_2 \cap \L_1} (w^{(r)}_{k_r})^2}\round{\prod_{r  \in K_2 \cap \L_2} (w^{(r)}_{k_r})^2}.
\end{align}
Now we consider the first summation in Eq.~\eqref{eqn:ewqn_messproduct}. Note that the following holds
\begin{align*}
    K_1 \cap \I_- = \bigsqcup_{u\in E_1}(G_u \cap \I_-).
\end{align*}
For each block $G_u$ with $u\in E_1$, we distribute the weight 
    $\prod_{r \in G_u \cap \I_-} (w^{(r)}_{k_r})^2$ equally to all $(w^{(r')}_{k_{r'}})^2$ with $r' \in G_u \cup \I_-$.
We thus obtain
\begin{align}\label{eqn:weight_assign_2}
    &~~~\sum_{k_r:\, r\in K_1} \round{\prod_{r  \in K_1 \cap \I_-} (w^{(r)}_{k_r})^2}(b_{1})_{k_1}a^1_{k_1,i_1} a^2_{i_1,k_2}  \ldots  a^{\emp+1}_{i_{\emp},k_{\emp+1}} (b_{2})_{k_{\emp+1}} \notag\\
    &=\sum_{k_r:\, r\in K_1}\round{(b_1)_{k_1}P_{\pi_2(1)}^{\frac{1}{2}}} \round{ a^1_{k_1,i_1}P_{\pi_2(1)}^{\frac{1}{2}}} \round{a^2_{i_1,k_2} P_{\pi_2(2)}^{\frac{1}{2}}} \round{ a^2_{k_2,i_2} P_{\pi_2(2)}^{\frac{1}{2}} } \ldots \round{a^{\emp+1}_{i_{\emp},k_{\emp+1}} P_{\pi_2(\emp+1)}^{\frac{1}{2}}} \round{(b_{2})_{k_{\emp+1}}P_{\pi_2(\emp+1)}^{\frac{1}{2}}}.
\end{align}
Plugging Eq.~\eqref{eqn:weight_assign_1} and \eqref{eqn:weight_assign_2} into Eq.~\eqref{eqn:ewqn_messproduct} yields:
\begin{align*}
   & \underbrace{\sum_{k_r:\, r\in K_1}\round{(b_1)_{k_1} P_{\pi_2(1)}^{\frac{1}{2}}} \round{ a^1_{k_1,i_1}P_{\pi_2(1)}^{\frac{1}{2}}} \round{a^2_{i_1,k_2} P_{\pi_2(2)}^{\frac{1}{2}}} \round{ a^2_{k_2,i_2} P_{\pi_2(2)}^{\frac{1}{2}} } \ldots \round{a^{\emp+1}_{i_{\emp},k_{\emp+1}} P_{\pi_2(\emp+1)}^{\frac{1}{2}}} \round{(b_2)_{k_{\emp+1}}P_{\pi_2(\emp+1)}^{\frac{1}{2}}}}_{:=\mathbf{S}_c}.\\
   &~~~\cdot\underbrace{\sum_{k_r:\, r\in K_2\cap \L_1} \round{\prod_{r  \in K_2 \cap \L_1} (w^{(r)}_{k_r})^2}\round{\prod_{r  \in K_2 \cap \L_2} (w^{(r)}_{k_r})^2}}_{:=\mathbf{S}_b}.
\end{align*}
We now bound $\E[\mathbf{S}_c]$  and $\mathbf{S}_b$ beginning with the former. Note that $ \mathbf{S}_c = \sum_{k_r:\, r\in K_1} [\cdots]$ only depends on indices $k_r$ with $r\in \I^+$. We explicitly write out all the feature indices $k_1,\ldots,k_{\emp+1}$ without applying partition, but we group the features based on partition of sample indices, i.e., $\pi_1 = (T_1,\ldots,T_{q})$. We assign new sample indices to each block $(i_1,\ldots,i_q)$. Then we can obtain
\begin{align}
    \mathbf{S}_c = \sum_{\substack{
k_r\in \S_{r},\\
r \in [\emp+1]
}}
(b_1)_{k_1}(b_2)_{k_{\emp+1}}P_{\pi_2(1)}^{\frac{1}{2}} P_{\pi_2(\emp+1)}^{\frac{1}{2}}  
\;
\prod_{j=1}^{q}
\;
\prod_{r\,\in\,T_j\cap [\emp]}
\;
\round{
a^{\,r}_{\;i_j,\,k_{r}}a^{\,r+1}_{\;i_j,\,k_{r+1}} \cdot P^{\frac{1}{2}}_{\pi_2(r)} P^{\frac{1}{2}}_{\pi_2(r+1)}
}.
\end{align}

\paragraph{Iterative summation over feature indices.} Next we iteratively take the summation over the feature indices i.e., $\sum_{\substack{
k_1,\dots,k_{\emp+1}
}}$  to the right side of $\prod_{j=1}^{q}$. To this end, observe the some blocks $T_i$ might have an empty intersection with $[\emp]$. Therefore without loss of generality we suppose that $T_1\cap \I_+,\ldots, T_e\cap \I_+$ are nonempty and cover $\I_+$ for some $e\in [q]$. We now define a sequence of disjoint sets of feature indices $R_1,...,R_e$ such that $R_1\sqcup R_2\sqcup \ldots \sqcup R_e = [\emp+1]$ holds. We iteratively remove $T_1$ from the graph along with all connected edges (ignoring orientation),  then $T_2$, and so on and so forth. 
When deletion of $T_j$ causes possibly new blocks $Q\cap \I_+$ with $Q\in\pi_2$ to become disconnected from the rest of the graph, we let $R_j$ consist of all the elements in these blocks $Q\cap \I_+$.  
Define now the accumulated set features $\overline R_j:=\bigcup_{i=1}^j R_i$.

As a concrete illustration of the definitions, consider a subpath:
\begin{align*}
    \to \underset{\underset{P_{\pi_2(3)}}{\rotatebox{90}{$\sim$}}}{k_3}\xrightarrow{3} i_4\xrightarrow{4} \underset{\underset{P_{\pi_2(4)}}{\rotatebox{90}{$\sim$}}}{k_4}\xrightarrow{4} i_5  \xrightarrow{5} \underset{\underset{P_{\pi_2(5)}}{\rotatebox{90}{$\sim$}}}{k_{5}}\xrightarrow{5} i_{6} \xrightarrow{6} \underset{\underset{P_{\pi_2(6)}}{\rotatebox{90}{$\sim$}}}{k_{6}}.
\end{align*}
Given $T_s = \{4,6\}$,  we delete the index $i_4$ and $i_6$ and the connected edges;  then the diagram becomes
\begin{align*}
    \to \underset{\underset{P_{\pi_2(3)}}{\rotatebox{90}{$\sim$}}}{k_3}~~~~~~~~~~~~\underset{\underset{P_{\pi_2(4)}}{\rotatebox{90}{$\sim$}}}{k_4}\xrightarrow{4} i_5  \xrightarrow{5} \underset{\underset{P_{\pi_2(5)}}{\rotatebox{90}{$\sim$}}}{k_{5}}~~~~~~~~~~~~
    \underset{\underset{P_{\pi_2(6)}}{\rotatebox{90}{$\sim$}}}{k_{6}}.
\end{align*}
After the deletion, node $k_6$ becomes isolated. Suppose $G_{\pi_2(6)} = \{6\}$ is a singleton. Then the block $G_{\pi_2(6)}$ is disconnected from the rest of the graph, therefore we set $R_s=\{6\}$.

To simplify notation, define the product along the edge $k_r\to i_j\to k_{r+1}$ as follows:
\begin{align}
    H_{j,r} := a^{\,r}_{\;\,k_{r},i_{j}}a^{\,r+1}_{\;i_{j},\,k_{r+1}} \cdot P^{\frac{1}{2}}_{\pi_2(r)} P^{\frac{1}{2}}_{\pi_2(r+1)}.
\end{align}
Note the  $H_{j,r}$ depends on $i_j$, $k_r$, and $k_{r+1}$.
For $s = 1,2,\ldots,e$, we define the quantity
\begin{align}\label{eq:def_v_s}
 \displaystyle   V_s = \begin{cases}
      \displaystyle  \sum_{k_t: \displaystyle t\in \overline{R}_s} \round{(b_1)_{k_{1}} P_{\pi_2(1)}^{\frac{1}{2}}}^{\mathbbm{1}\{1\in \overline{R}_s\}}  \round{(b_2)_{k_{\emp+1}} P_{\pi_2(\emp+1)}^{\frac{1}{2}}}^{\mathbbm{1}\{\emp+1 \in \overline{R}_s\}}  \prod_{j=1}^s  \E\brac{
\displaystyle\prod_{r\,\in\, T_j \cap [\emp]}
\;
H_{j,r},}~~&{\rm if}~ \displaystyle \overline{R}_s\neq \emptyset;\\
 \prod_{j=1}^s  \E\brac{
\displaystyle\prod_{r\,\in\, T_j \cap [\emp]}
\;
H_{j,r},}, &{\rm otherwise}.
    \end{cases} 
\end{align}
Note that the value $V_s$ depends on all $k_t$ with $t\in [\emp+1]\setminus \overline{R}_s$ since the sum in \eqref{eq:def_v_s} is taken over $t\in \overline{R}_s$. 
Clearly, equality $\E[\mathbf{S}_c] =  V_e$ holds and therefore it suffices now to bound $V_e$. We do so by appealing to the following claim, which shows that $V_s$ admits a recursive relation.  To simplify notation, we set $\hat b_1=b_1\cdot P_{\pi_2(1)}^{1/2}$ and similarly $\hat b_2=b_1\cdot P_{\pi_2(\emp+1)}^{1/2}$.
\begin{claim}
The following estimates hold for $s = 2,\ldots,e$:
\begin{align}\label{eqn:recurs}
        V_s = \begin{cases}
       \sum_{k_t: t\in R_s} (\hat b_1)_{k_1}^{\mathbbm{1}\{1\in {R}_s\}}(\hat b_2)_{k_{\emp+1}}^{\mathbbm{1}\{\emp+1\in {R}_s\}} V_{s-1}\cdot  \E\brac{
\prod_{r\,\in\,T_s\cap[\emp]}
\;
H_{s,r}},~~&{\rm if}~ R_s\neq \emptyset;\\
 V_{s-1} \cdot \E\brac{
\prod_{r\,\in\,T_s\cap[\emp]}
\;
H_{s,r}}, &{\rm otherwise}.
    \end{cases} 
\end{align}
\end{claim}
\begin{proof}
We present the proof for $ R_s \neq \emptyset$. The proof for the other case follows similarly.
Now an elementary argument shows that the equation
\begin{align}
&\sum_{k_t: t\in \overline{R}_s} (\hat b_1)_{k_1}^{\mathbbm{1}\{1\in \overline{R}_s\}}(\hat b_2)_{k_{\emp+1}}^{\mathbbm{1}\{\emp+1\in \overline{R}_s\}} [\cdots]\notag\\
&=\sum_{k_t: t\in R_s}\sum_{k_t: t\in \overline{R}_{s-1}}  (\hat b_1)_{k_1}^{\mathbbm{1}\{1\in \overline{R}_s\}}(\hat b_2)_{k_{\emp+1}}^{\mathbbm{1}\{\emp+1\in \overline{R}_s\}} [\cdots]\notag\\
&=\sum_{k_t: t\in R_s}(\hat b_1)_{k_1}^{\mathbbm{1}\{1\in {R}_{s}\}}(\hat b_2)_{k_{\emp+1}}^{\mathbbm{1}\{\emp+1\in {R}_{s}\}} \sum_{k_t: t\in \overline{R}_{s-1}}  (\hat b_1)_{k_1}^{\mathbbm{1}\{1\in \overline{R}_{s-1}\}}(\hat b_2)_{k_{\emp+1}}^{\mathbbm{1}\{\emp+1\in \overline{R}_{s-1}\}} [\cdots],\label{eqn:last_eqn_brak}
\end{align}
holds for any quantity $[\cdots]$. Plugging in $[\cdots]=\prod_{j=1}^s\E\brac{
\prod_{r\,\in\, T_j\cap [\emp]}
\;
H_{j,r}}$
on the left yields $V_s$. On the right side, we may further simplify the inner sum as follows.
Since the features in $T_s$ and $T_1,\ldots,T_{s-1}$ are independent, we can split the expectation and obtain
\begin{align*}
  &\sum_{k_t: t\in \overline{R}_{s-1}}  (\hat b_1)_{k_1}^{\mathbbm{1}\{1\in \overline{R}_{s-1}\}}(\hat b_2)_{k_{\emp+1}}^{\mathbbm{1}\{\emp+1\in \overline{R}_{s-1}\}} \prod_{j=1}^s\E\brac{
\prod_{r\,\in\, T_j\cap [\emp]}
\;
H_{j,r}}\\
&= \sum_{k_t:t\in  \overline{R}_{s-1}} (\hat b_1)_{k_1}^{\mathbbm{1}\{1\in \overline{R}_{s-1}\}}(\hat b_2)_{k_{\emp+1}}^{\mathbbm{1}\{\emp+1\in \overline{R}_{s-1}\}}\left(\prod_{j=1}^{s-1}\E\brac{\prod_{r\in T_j \cap [\emp]} H_{j,r}} \cdot \E\brac{\prod_{r\in T_s\cap[\emp]} H_{s,r}}\right) \\
&= \underbrace{\left(\sum_{k_t:t\in  \overline{R}_{s-1}} (\hat b_1)_{k_1}^{\mathbbm{1}\{1\in \overline{R}_{s-1}\}}(\hat b_2)_{k_{\emp+1}}^{\mathbbm{1}\{\emp+1\in \overline{R}_{s-1}\}}\prod_{j=1}^{s-1}\E\brac{\prod_{r\in T_j \cap [\emp]} H_{j,r}} \cdot \right)}_{=V_{s-1}}\E\brac{\prod_{r\in T_s\cap[\emp]} H_{s,r}}.
\end{align*}
where the last inequality follows from the key fact that 
$H_{s,r}$ with $r\in T_s\cap [\emp]$ are independent of $k_t$ with $t\in \overline{R}_{s-1}$ and therefore $\E\brac{\prod_{r\in T_s\cap[\emp]} H_{s,r}}$ can be pulled out of the inner sum. Next, if $1\notin \overline{R}_{s}$ and $\emp+1\notin \overline{R}_{s}$, then $Q=V_{s-1}$. 
Thus plugging this expression back into \eqref{eqn:last_eqn_brak} completes the proof of \eqref{eqn:recurs}.
\end{proof}

\paragraph{Bounding $V_s$.} 

We now inductively bound $V_s$ as follows. Suppose for the moment that $R_s\neq \emptyset$ and consider the expression in \eqref{eqn:recurs}. Applying absolute values to both sides and using the triangle inequality gives:
\begin{align}\label{eqn:key_bound_neededhere}
    |V_s|\leq \sum_{k_t: t\in R_s} (|\hat b_1|)_{k_1}^{\mathbbm{1}\{1\in {R}_s\}}(|\hat b_2|)_{k_{\emp+1}}^{\mathbbm{1}\{\emp+1\in {R}_s\}} |V_{s-1}|\cdot  \E\brac{
\prod_{r\,\in\,T_s\cap[\emp]}
\;
a^{\,r}_{\;\,k_{r},i_{s}}a^{\,r+1}_{\;i_{s},\,k_{r+1}} \cdot P^{\frac{1}{2}}_{\pi_2(r)} P^{\frac{1}{2}}_{\pi_2(r+1)}}.
\end{align}
Notice that the product is taken over the same sample index $i_s$ but different feature indices $k_r$. Consequently, the product coincides with a nonngetive multiple of a Fourier-Walsh monomial on the hypercube and therefore its expectation is nonnegative. Replacing $V_{s-1}$ by $\displaystyle\max_{k_t:t\in \I_+ }|V_{s-1}|$ clearly yields a valid upper-bound on  $|V_s|$.

Now, we aim to apply Proposition~\ref{prop:fourier_basis_product_card_2} along with Remark~\ref{rem:needed_remark}. 
Observe that due to the collapsing procedure, crucially each index $k_r$ can appear at most once in the product in \eqref{eqn:key_bound_neededhere}.
  Thus, we may treat the product of features that do not depend on any $k_t\in R_s$ as $x^{S^*}$, and each $a_{i_j,k_r}$ with $k_r\in R_s$ as $x^{S_t}$.  Since each summand in Eq.~\eqref{eq:card_bound_2} is non-negative, for each $S_t$ with $t\in[q]$, we can multiply $x^{S_t}$ by a positive factor, then the product of these factors will appear in the upper bound.
Therefore, Proposition~\ref{prop:fourier_basis_product_card_2} directly yields the bound
\begin{align*}
        |V_{s}| &= \round{\max_{k_t:t\in \I_+ }|V_{s-1}|}\cdot O_d(d^{\sum_{t \in R_s} p_{t}/2}) \cdot \round{\norm{b_1}_2 \norm{w^{(1)}}_\infty  P_{\pi_2(1)}^{\frac{1}{2}}d^{-p_1/2}}^{\mathbbm{1}\{1\in {R}_s\}} \\
        &\cdot \round{\norm{b_2}_2 \norm{w^{\emp+1}}_\infty P_{\pi_2(\emp+1)}^{\frac{1}{2}}d^{-p_{\emp+1}/2}}^{\mathbbm{1}\{\emp+1 \in {R}_s\}} 
\cdot\round{\prod_{r\,\in\,T_s}\norm{w^{{r}}}_\infty\norm{w^{{r+1}}}_\infty P^{\frac{1}{2}}_{\pi_2(r)}P^{\frac{1}{2}}_{\pi_2(r+1)}},
\end{align*}
where we use the fact that each $a^r_{i,k_r}$ is scaled with a weight bounded by $\norm{w_r}_\infty$.

Iterating the bound on $|V_s|$, we conclude
\begin{align}\label{eq:e_s_c_final}
\E[\mathbf{S}_c]=   O_d(d^{\sum_{t=2}^{\emp} p_{t}/2})\cdot  \norm{b_1}_2 \norm{b_2}_2 \norm{w^{(1)}}_\infty \norm{w^{(\emp+1)}}_\infty  \prod_{j=2}^{\emp}\norm{w^{(j)}}_\infty^2 \cdot 
\prod_{\substack{r \in K_1 \cap \I_-} } \norm{w^{(r)}}_\infty^2,
\end{align}
where  $\prod_{\substack{r \in K_1 \cap \I_-} } \norm{w^{(r)}}_\infty^2$ comes from the distributed weights $P$.

Now we bound $\mathbf{S}_b$ as follows:
\begin{align}\label{eq:s_b_final}
    \mathbf{S}_b &=
    \sum_{\substack{k_r: r\in K_2\cap \L_1}} \prod_{r\in K_2\cap \L_1} (w^{(r)}_{k_r})^2 \prod_{r\in K_2 \cap \L_2} (w^{(r)}_{k_r})^2\notag\\
    &\leq \prod_{r \in K_2 \cap \L_2} \norm{w^{(r)}}_\infty^2  \sum_{\substack{k_r: r\in K_2\cap \L_1}} \prod_{r\in K_2 \cap \L_1} (w^{(r)}_{k_r})^2 \notag\\
    &\leq  \prod_{r \in K_2 \cap \L_2} \norm{w^{(r)}}_\infty^2 \cdot \prod_{r\in K_2\cap \L_1} \norm{w^{(r)}}^2_\infty \cdot  O_d(d^{\sum_{t\in \L_1\cap K_2 }p_t}).
\end{align}
We now split $\prod_{\substack{r \in K_1 \cap \I_-} } \norm{w^{(r)}}_\infty^2$ into the product of $\prod_{\substack{r \in K_1 \cap \L_1} } \norm{w^{(r)}}_\infty^2$ and $\prod_{\substack{r \in K_1 \cap \L_2} } \norm{w^{(r)}}_\infty^2$ in Eq.~\eqref{eq:e_s_c_final}. 
Therefore, combining Eq.~\eqref{eq:e_s_c_final} and \eqref{eq:s_b_final} yields:
\begin{equation}\label{eq:fin_boundcb}
    \E[\mathbf{S}_c] \mathbf{S}_b \leq O_d(d^{\sum_{t=2}^{\emp} p_{t}/2})\cdot  \norm{b_1}_2 \norm{b_2}_2 \norm{w^{(1)}}_\infty \norm{w^{(\emp+1)}}_\infty  \prod_{j=2}^{\emp}\norm{w^{(j)}}_\infty^2 \cdot \prod_{r\in \L_2} \norm{w^{(r)}}_\infty^2 \prod_{r\in \L_1} \norm{w^{(r)}}^2_\infty \cdot  O_d(d^{\sum_{t\in K_2\cap \L_1}p_t}).
\end{equation}

Finally, recall that to get a bound on $\E[\mathbf{S}_\pi]$ we can multiply \eqref{eq:fin_boundcb} by $n^{q}$. The following claim gives a bound for $ q$.
\begin{claim}
The estimate $q \leq \emp + |\L_2|$ holds.
\end{claim}
\begin{proof}
We split the set $[q]$ into two non-overlapping sets:
\begin{align*}
D_1&:=\{i~\vert~T_i \cap [\emp] \neq \emptyset\}\qquad \textrm{and}\qquad
D_2:=\{i~\vert~T_i \cap [\emp] = \emptyset\}.
\end{align*}
Recall that the cardinality of $T_i\cap [\emp]$ for any $i\in D_1$ is at least two and therefore the estimate $|D_1| \leq \emp$ holds. Now we consider the size of $D_2$. 
An argument completely analogous to Claim~\ref{claim:g_u_l_1} shows that for any $i\in D_2$, the intersection $T_i\cap \L_2$ is nonempty. 
Therefore, the estimate $|D_2| \leq |\L_2|$ holds. Combining our two upper bounds on $D_1$ and $D_2$ we conclude
$q = |D_1| + |D_2|\leq \emp +|\L_2|$, as claimed.
\end{proof}

Multiplying \eqref{eq:fin_boundcb} by $n^{q}\leq n^{\emp + |\L_2|}$ we  obtain the final bound:
\begin{align}\label{eq:s_pi_final}
    \E[\mathbf{S}_\pi] &\leq   \norm{b_1}_2 \norm{b_2}_2 
 \cdot O_d(d^{\sum_{t=2}^{\emp} p_{t}/2})\cdot n^{\emp}\norm{w^{(1)}}_\infty\norm{w^{(\emp+1)}}_\infty\prod_{j=2}^{\emp}\norm{w^{(j)}}_\infty^2 \notag\\
 &~~~\cdot \prod_{r\in \L_2} n\norm{w^{(r)}}_\infty^2\cdot  \prod_{r\in \L_1} \norm{w^{(r)}}^2_\infty \cdot  O_d(d^{\sum_{t\in \L_1 \cap K_2}p_t}).
\end{align}

\paragraph{Completing the proof.} %
We first verify the bound \eqref{eqn:labeled_hard_mat} for $t\in \{2,\ldots,\emp\}$. To this end, the estimate \eqref{eq:s_pi_final} directly implies 
\begin{align*}
     \E[\mathbf{S}_\pi] = \norm{b_1}_2 \norm{b_2}_2 \round{\prod_{t=2}^{\emp}O_d(d^{p_t/2})\cdot n^{1/2}\norm{w^{(t)}}_\infty^2} n^{1/2} \norm{w^{(1)}}_\infty\norm{w^{{(\emp+1)}}}_\infty \cdot O_d(1),
\end{align*}
where we use the fact that $\prod_{r\in \L_2} n\norm{w^{(r)}}_\infty^2$ and   $\prod_{r\in \L_1} \norm{w^{(r)}}^2_\infty \cdot  O_d(d^{\sum_{t\in \L_1 \cap K_2}p_t})$ are of the order $O_d(1)$ since by assumption $\norm{w^{(r)}}_{\infty} = O_d(n\half \wedge d^{-p_r/2})$ holds.

Recall that we have $n\norm{w^{(t)}}_\infty^2 = O_d(1)$. Therefore, by switching any single $O_d(d^{p_t/2})$ inside the parenthesis with $n^{1/2}$ outside the parenthesis, each product inside the parenthesis is $O_d(1)$ and therefore we
obtain
\begin{align}\label{eq:s_pi_bound_1}
     \E[\mathbf{S}_\pi] = \norm{b_1}_2 \norm{b_2}_2 \cdot O_d(d^{p_t/2})\cdot \norm{w^{(1)}}_\infty\norm{w^{\emp+1}}_\infty,
\end{align}
for any $t\in \{2,\ldots, \emp\}$.   Noting that $w^{{(1)}}$ and $w^{(\emp+1)}$ are in the blocks as the pre-labeled first and last indices respectively, which completes the proof of \eqref{eqn:labeled_hard_mat}.

Next, consider any index $r\in \mathcal{L}_1$. Then from \eqref{eq:s_pi_final} we obtain  the bound:
\begin{align*}
    \E[\mathbf{S}_\pi] &=  \norm{b_1}_2 \norm{b_2}_2 
 \cdot\round{\prod_{t=2}^{\emp}O_d(d^{p_t/2})\cdot n^{1/2}\norm{w^{(t)}}_\infty^2} n^{1/2} \norm{w^{(1)}}_\infty\norm{w^{(\emp+1)}}_\infty\\
 &~~~\cdot \prod_{r\in \L_2} n\norm{w^{(r)}}_\infty^2\cdot  \prod_{r\in \L_1} \norm{w^{(r)}}^2_\infty \cdot  O_d(d^{\sum_{t\in \L_1 \cap K_2}p_t})\\
 &=\norm{b_1}_2 \norm{b_2}_2 
 \cdot n^{1/2}\norm{w^{(1)}}_\infty\norm{w^{\emp+1}}_\infty  \prod_{r\in \L_1} O_d(\norm{w^{(r)}}^2_\infty d^{p_r}),
\end{align*}
where the first equality follows from rearranging the factors and the second equality uses the fact that $\round{\prod_{t=2}^{\emp}O_d(d^{p_t/2})\cdot n^{1/2}\norm{w^{(t)}}_\infty^2}$ and $  \prod_{r\in \L_2} n\norm{w^{(r)}}_\infty^2$ are of the order $O_d(1)$.
Since for any $r$ we have $\norm{w^{(r)}}_\infty n^{1/2} = O_d(1)$, it follows that for any $t\in \mathcal{L}_1$ we may replace the term $\norm{w^{(r)}}^2_\infty d^{p_t}$ in the product by $\norm{w^{(r)}}^2_\infty n=O_d(1)$ and multiply outside the product by $d^{p_t}/n$. The result product is $O_d(1)$ and therefore we conclude:
\begin{align}\label{eq:s_pi_bound_2}
     \E[\mathbf{S}_\pi] = \norm{b_1}_2 \norm{b_2}_2 O_d(d^{p_t})n\half\cdot \norm{w^{(1)}}_\infty\norm{w^{(\emp+1)}}_\infty,
\end{align}
for any $t\in \L_1$. Again noting that $w^{(1)}$ and $w^{\emp+1}$ are in the blocks as the pre-labeled first and last indices respectively, which completes the proof of \eqref{eqn:labeled_hard_mat}.

Finally, let $i,j\in [m+1]$ be the new indices that correspond to the original indices $1$ and $p+1$, respectively. As the final case to consider, fix any index $r\in \L_2\cup\{1,\emp+1\}$ satisfying $\pi_2(r)\neq \pi_2(i)$ and $\pi_2(r)\neq \pi_2(j)$.
Define the block index $u:=\pi_2(r)$. If $G_u$ intersects $\{2,\ldots,\emp\}$, then \eqref{eq:s_pi_bound_1} implies the same bound for $p_r$. Therefore we may suppose that this is not the case. Taking into account the equalities $\pi_2(1)=\pi_2(i)$ and $\pi_2(\emp+1)=\pi_2(j)$, we deduce that $G_u$ contains neither $1$ nor $\emp+1$. Thus the inclusion $G_u\subset \mathcal{L}_{1}\cup \mathcal{L}_2$ holds. Since no block of $\pi_2$ is contained in $\mathcal{L}_2$ we deduce that there exists some $r'\in \mathcal{L}_1$ satisfying $\pi_2(r)=\pi_2(r')$. Since we have already proved the result for all elements $\mathcal{L}_1$, the proof is complete in this case as well.

Finally, throughout the proof we assumed that $k_1$ has not collapsed with $k_{m+1}$. Had this in fact happened, then we have $\emp=1$ and all the intermediate features collapsed with the neighboring features. Then Eq.~\eqref{eq:s_pi_final} reduces to 
\begin{equation}\label{eqn:final_simplifiedest}
    \E[\mathbf{S}_\pi] =  \norm{b_1}_2\norm{b_2}_2
 \cdot \prod_{r\in \L_2} n\norm{w^{(r)}}_\infty^2\cdot  \prod_{r\in \L_1} O_d\round{\norm{w^{(r)}}^2_\infty \cdot d^{p_r}} .
\end{equation}
By exactly, the same argument as before, it suffices to prove the desired bound only for \eqref{eqn:labeled_hard_mat}  for  $t\in \L_1$. Let us find the new indices $i,j\in [m+1]$ that correspond to the original indices  $1$ and $m+1$. 
Clearly, we have $\pi_2(i)=\pi_2(j)$ since all the features collapsed. Consequently, either $i$ or $j$ lies in $\mathcal{L}_2$. Without loss of generality, therefore we may assume it is $i$. We now write 
$$\round{n\norm{w^{(i)}}^2_\infty}\round{\norm{w^{(t)}}^2_\infty d^{p_t}}=  \round{n\norm{w^{(t)}}^2_\infty}    \round{\norm{w^{(i)}}^2_\infty d^{p_t}}=O_d(1)\cdot \norm{w^{(i)}}^2_\infty d^{p_t}.$$
Appealing to \eqref{eqn:final_simplifiedest} we deduce
\begin{align*}
    \E[\mathbf{S}_\pi] &=  \norm{b_1}_2\norm{b_2}_2
 \cdot \prod_{r\in \L_2, r\neq i} n\norm{w^{(r)}}_\infty^2\cdot  \prod_{r\in \L_1, r\neq t} \norm{w^{(r)}}^2_\infty \cdot  O_d(d^{\sum_{t\in \L_1}p_t})\cdot \round{n \norm{w^{(t)}}_\infty^2}\norm{w^{(i)}}_{\infty}^2   \cdot O_d(d^{p_t})\\
 &= \norm{b_1}_2 \norm{b_2}_2\norm{w^{(i)}}^2_{\infty}   \cdot O_d(d^{p_t}),
\end{align*}
noting that the norm $\norm{w^{(i)}}_\infty = \norm{w^{(j)}}_\infty$  and $\norm{b_1} = \norm{b_2}=  1$ hold which completes the proof.

\section{Relaxing assumptions in Theorem~\ref{lemma:feature_product_tech}}
We complete the paper by recording another extension of the theorem---needed for an upcoming paper of the coauthors---where each $S_i\in \S_i$ is the union of distinct sets from multiple set families. More precisely, it may be the case that each set $S_i$ is a set family $\S_i$ may be written as 
$$S_i=\sqcup_{k=1}^{\ell} S_i^{[k]}$$
for some sets $S_i^{[k]}\subset \T_k$, where $\{\T_k\}_{k=1}^{\ell}$ form a partition of all the coordinates $[d]$ and  are potentially of sublinear size $\T_k=O_d(d^{s_k})$ with $s_k\in [0,1]$. In this case, one would expect that the ``effective degree'' of $\S_i$ is $\sum_{k} s_ip_i^{[k]}$ rather than the larger quantity $\sum_{k} p_i^{[k]}$. This is the content of the following theorem.
\begin{theorem}\label{lemma:feature_product_tech_new}
Suppose there exist non-overlapping sets $\T_1,\ldots, \T_\ell \subset [d]$ that satisfy $|\T_k| = \Theta(d^{s_k})$ for all $k\in[\ell]$ where $s_i\in[0,1]$. Fix the set families $ \S_1^{[k]}, \ldots  \S_{m+1}^{[k]}\in 2^{\T_k}$ for $k\in[\ell]$.
Consider collections of sets $\S_1,\ldots,\S_{m+1} \in 2^{[d]}$ and weights 
$w^{(i)}\in \R^{\S_i}_{+}$, where $S_i = \sqcup_{k=1}^\ell S_{i}^{[k]}$ for all $S_i\in \S_i$.  Define the matrix product
$$M= A_1^\top \Bigl(A_2 A_2^\top \Bigr) \cdots \Bigl(A_{m} A_{m}^\top \Bigr) A_{m+1},$$
where $A_i=X_{\mathcal{S}_i} \Diag(w^{(i)})$ are the scaled Fourier-Walsh matrices.
Assume  that the following regularity conditions hold for all indices $i,j\in [m+1]$:
\begin{enumerate}
    \item\label{it:deg_bound_new} {\bf (degree bound)} The inequalities $| S^{[k]}_i|\leq  p_i^{[k]}$ hold for all $ S_i^{[k]}\in  \S_i^{[k]}$,
    \item\label{it:triv_intersec_new} {\bf (trivial intersection)} Whenever $\S_i$ intersects $\S_j$, the equality $\S_i=\S_j$ holds,
    \item\label{it:small_weights_new} {\bf (small weights)} The weights satisfy 
    $w^{(i)}= O_d(n^{-1/2}\wedge d^{-\round{\sum_{k=1}^\ell s_k \cdot p_i^{[k]}}/2})$ for all $i\in [m+1]$.
\end{enumerate}
    Then the estimate
    \begin{equation}\label{eqn:labeled_hard_mat_new}
    \left\|\E M \right\|_{\rm op}
    = O_d \left(d^{\sum_{k=1}^\ell s_k \cdot p_j^{[k]}}\right)\|w^{(1)}\|_\infty \|w^{(m+1)}\|_\infty,
    \end{equation}
    holds for any index $j \in [m+1]$ such that $\S_{j}$ is distinct from $\S_{1}$ and $\S_{{m+1}}$.
\end{theorem}

The proof is nearly identical. One only needs to note that when bounding $V_s$ (Eq.~\eqref{eqn:key_bound_neededhere}), we decompose $x^{S_i}$ as the product $\prod_{k=1}^\ell x^{S_i^{[k]}}$ and apply Proposition~\ref{prop:fourier_basis_product_card_2} separately to each component $S_i^{[k]}$. This decomposition is valid because $S_i^{[k]}$ for $k\in[\ell]$ are disjoint sets, which allows us to factorize the expectation. Consequently, we may replace the factor $d^{p_i/2}$ by $d^{\round{\sum_{k=1}^\ell s_k \cdot p_i^{[k]}}/2}$ in all the bounds in  Proposition~\ref{prop:fourier_basis_product_card_2}. For example, 
Eq.~\eqref{eq:card_bound_3_group} becomes 
\begin{align}
     &~~\sum_{(S_{1},\ldots,S_{q}) \in \S_\pi}b_{S_q}\E\brac{x^{S_1}x^{S_2}\ldots x^{S_q}x^{S^*}} = \sum_{(S_{1},\ldots,S_{q}) \in \S_\pi} \prod_{k=1}^\ell b_{ S_q^{[k]}} \E\brac{x^{ S_1^{[k]}}x^{ S_2^{[k]}}\ldots x^{ S_q^{[k]}}x^{ S_*^{[k]}}}\notag \\
    &= \norm{b}_2^\ell\cdot \prod_{k=1}^{\ell}O_{d}((d^{s_k})^{(\sum_{t=1}^{q-1}  p_t^{[k]} )/2}),
\end{align}
where we decompose $S^*$ into $\sqcup S_*^{[k]}$ corresponding to $\T_k$.

Correspondingly, the bound obtained by iterating the bound on $|V_s|$ i.e., Eq.~\eqref{eq:e_s_c_final} becomes
\begin{align*}
\E[\mathbf{S}_c]=   O_d(d^{\sum_{t=2}^{\emp} \sum_{k=1}^\ell (s_k \cdot p_t^{[k]})/2})\cdot  \norm{b_1}_2^\ell \norm{b_2}_2^\ell \norm{w^{(1)}}_\infty \norm{w^{(\emp+1)}}_\infty  \prod_{j=2}^{\emp}\norm{w^{(j)}}_\infty^2 \cdot 
\prod_{\substack{r \in K_1 \cap \I_-} } \norm{w^{(r)}}_\infty^2.
\end{align*}
With the weights accommodated with the new features, the remainder of the proof is the same, and the factor $ d^{\sum_{k=1}^\ell s_k \cdot p_j^{[k]}}$ appears in the final bound.

\printbibliography

\appendix
\section{Proof of Equality~\eqref{eqn:intro_eqn_moment}}\label{proof:intro_eqn_moment}
Fix a vector $b\in \R^{\S}$. Then elementary algebraic manipulations show the equality:
\begin{align*}
    \frac{1}{n^2}b\tran \E\brac{X_\S\tran X_{\S'} X_{\S'}\tran X_\S}b=\frac{1}{n^2} \sum_{i,j\in[n]}\sum_{S_1\in \S, S_2\in \S, S'\in\S'} b_{S_1} \E\brac{x^{(i)}_{S_1} x^{(i)}_{S'} x^{(j)}_{S'}x^{(j)}_{S_2}} b_{S_2}.
\end{align*}
Since $\S\cap \S' = \emptyset$, all summands corresponding to $i\neq j$ are zero. Thus, the right-hand-side becomes
\begin{align*}
        \frac{1}{n^2} \sum_{i\in[n]}\sum_{S_1\in \S, S_2\in \S, S'\in\S'} b_{S_1} \E\brac{x^{(i)}_{S_1} \underbrace{x^{(i)}_{S'} x^{(i)}_{S'}}_{=1}x^{(i)}_{S_2}} b_{S_2} &=   \frac{|\S'|}{n^2} \sum_{i\in[n]}\sum_{S_1\in \S, S_2\in \S} b_{S_1} \E\brac{x^{(i)}_{S_1} x^{(i)}_{S_2}} b_{S_2}= \frac{|\S'|}{n}\norm{b}_2^2,
\end{align*}
where the second equality follows from the identity $\E\brac{x^{(i)}_{S_1} x^{(i)}_{S_2}} = \mathbbm{1}{\left\{S_1=S_2\right\}}$. Since $b$ is arbitrary, the claimed estimate \eqref{eqn:intro_eqn_moment} follows.

\end{document}

%% file: math_commands.tex
\usepackage{amsmath,amsfonts,bm}

\def\1{\bm{1}}

\DeclareMathAlphabet{\mathsfit}{\encodingdefault}{\sfdefault}{m}{sl}
\SetMathAlphabet{\mathsfit}{bold}{\encodingdefault}{\sfdefault}{bx}{n}

\newcommand{\E}{\mathbb{E}}

\newcommand{\R}{\mathbb{R}}
\newcommand{\emp}{{\widetilde{m}}}

\renewcommand\L{{\mathcal{L}}}
\renewcommand\S{{\mathcal{S}}}

\renewcommand\R{{\mathbb{R}}}

\newcommand\Diag{{\mathsf{Diag}}}

\usepackage{physics}

\newcommand\I{{\mathcal{I}}}

\newcommand\T{{\mathcal{T}}}

\renewcommand{\norm}[1]{\left\|#1\right\|}
\newcommand{\snorm}[1]{\left\|#1\right\|_{\mathrm{op}}}

\newcommand{\round}[1]{\left(#1\right)}
\newcommand{\brac}[1]{\left[#1\right]}
\renewcommand{\abs}[1]{\left|#1\right|}

\newtheorem{lemma}{Lemma}
\newtheorem{theorem}{Theorem}
\newtheorem{corollary}{Corollary}
\newtheorem{claim}{Claim}
\newtheorem{proposition}{Proposition}

\newtheorem{remark}{Remark}

\newcommand{\tran}{^\top}

\newcommand{\half}{^{-\frac{1}{2}}}